\documentclass[]{amsart}
\usepackage{amssymb,amscd}
 
\usepackage{amsmath}
\usepackage{amsxtra}
\usepackage{amsfonts} 
\usepackage{amsthm}
\usepackage{here}
\usepackage{epsfig,subfigure}
\usepackage{psfrag}
\newlength{\halfbls}

\usepackage{hyperref}

\newtheorem{theorem}{Theorem}[section]
\newtheorem{prop}[theorem]{Proposition} 
\newtheorem{cor}[theorem]{Corollary}

\newtheorem*{exs}{Examples}
\newtheorem{lemma}[theorem]{Lemma}

\theoremstyle{definition}

\newtheorem{rem}[theorem]{Remark} 

\theoremstyle{remark}

\renewcommand{\div}{{\rm div}}
\newcommand{\ZZ}{\mathbb{Z}}

\newcommand{\CC}{\mathbb{C}}

\newcommand{\PP}{\mathbb{P}}

\newcommand{\LL}{\mathbb{L}}

\newcommand{\MM}{\mathbb{M}}

\newcommand{\WW}{\mathbb{W}}

\newcommand{\QQ}{\mathbb{Q}}
\newcommand{\RR}{\mathbb{R}}

\newcommand{\proj}{{\mathbb P}}

\newcommand{\Gal}{{\rm Gal}}
\newcommand{\Lin}{{\rm Lin}}

\newcommand{\SL}{{\rm SL}}

\newcommand{\Id}{{\rm Id}}
\newcommand{\moduli}[1][g]{{\mathcal M}_{#1}}
\newcommand{\omoduli}[1][g]{{\Omega \mathcal M}_{#1}}
\newcommand{\barmoduli}[1][g]{{\overline{\mathcal M}}_{#1}}

\newcommand{\cX}{\mathcal{X}}
\newcommand{\bcX}{\overline{\mathcal{X}}}

\newcommand{\cL}{\mathcal{L}}

\newcommand{\cO}{\mathcal{O}}
\newcommand{\cQ}{\mathcal{Q}}
\newcommand{\ulw}{\underline{w}}

\newcommand{\OO}{\mathcal O}

\newcommand{\pomoduli}[1][g]{{\proj\Omega\mathcal M}_{#1}}

\newcommand{\obarmoduli}[1][g]{{\Omega\overline{\mathcal M}}_{#1}}

\newcommand{\pobarm}[1][g]{\proj\Omega\overline{\mathcal M}_{#1}}

\newcommand{\Mg}{\overline{\mathcal M}_{g}}

\newcommand{\Sgo}{\overline{\mathcal S}^{-}_g}
\newcommand{\Sge}{\overline{\mathcal S}^{+}_g}
\newcommand{\Sg}{\overline{\mathcal S}_g}
\newcommand{\Zg}{\overline{Z}_g}

\newcommand{\bC}{{\overline{C}}}
\newcommand{\bB}{{\overline{B}}}

\newcommand{\Nfold}{{\rm Nfold}}
\newcommand{\other}{{\rm other}}
\newcommand{\hyp}{{\rm hyp}}
\newcommand{\nh}{{\rm non-hyp}}
\newcommand{\odd}{{\rm odd}}
\newcommand{\even}{{\rm even}}

\newcommand{\tr}{{\rm tr}}

\newcommand{\fun}{{\rm fun}}

\renewcommand{\Im}{\IM}

\newcommand{\teichm}{{Teich\-m\"uller\ }}

\newcommand{\irr}{{\rm irr}}
\newcommand{\rel}{{\rm rel}}

\DeclareMathOperator{\IM}{Im}

\DeclareMathOperator{\Pic}{Pic}

\title[Non-varying sums of Lyapunov exponents]{Non-varying sums of Lyapunov exponents of 
abelian differentials in low genus}

\author{Dawei Chen}

\address{Department of Mathematics, Boston College, Chestnut Hill, MA 02467, USA}

\email{dawei.chen@bc.edu}

\author{Martin M\"{o}ller}

\address{Institut f\"{u}r Mathematik, Goethe-Universit\"{a}t Frankfurt, Robert-Mayer-Str. 6-8, 60325 Frankfurt am Main, Germany}

\email{moeller@math.uni-frankfurt.de}

\thanks{During the preparation of this work the first author is partially supported by the 
NSF grant DMS-1101153 (transferred as DMS-1200329). The second author is partially supported
by the ERC-StG 257137. }

\begin{document}
\bibliographystyle{halpha}
 
\begin{abstract} We show that for many strata of Abelian differentials in low genus the sum of
Lyapunov exponents for the Teichm\"uller geodesic flow 
is the same for all \teichm curves in that
stratum, hence equal to the sum of Lyapunov exponents for
the whole stratum. This behavior is due to the disjointness 
property of \teichm curves with various geometrically defined 
divisors on moduli spaces of curves. 
\end{abstract}

\maketitle

\setcounter{tocdepth}{1}
\tableofcontents

\noindent

\section{Introduction}

Lyapunov exponents of dynamical systems are often hard to calculate
explicitly. For the \teichm geodesic flow on the moduli space of
Abelian differentials at least the {\em sum} of the positive Lyapunov
exponents is accessible for two cases. The moduli space decomposes
into various strata, each of which carries a finite invariant measure
with full support. For these measures the sum of Lyapunov exponents
can be calculated using \cite{emz} together with \cite{ekz}. On the other hand, the strata
contain many \teichm curves, e.g.\ those generated by square-tiled
surfaces. For \teichm curves an algorithm in \cite{ekz} calculates the sum
of Lyapunov exponents, of course only one \teichm curve at a time.
\par
On several occasions, one likes to have estimates, or even the 
precise values of Lyapunov exponents for all \teichm curves
in the same stratum simultaneously. For example, it is shown in \cite{delhub}
that Lyapunov exponents are responsible for the rate of diffusion
in the wind-tree model, where the parameters of the obstacle correspond
to picking a flat surface in a fixed stratum. One would like to 
know this escape rate not only for a specific choice of parameters
nor for the generic value of parameters but for {\em all} parameters.
\par
Zorich communicated to the authors, that, based on a limited number of 
computer experiments about a decade ago, Kontsevich and Zorich observed that 
the sum of Lyapunov exponents is non-varying among all the \teichm
curves in a stratum roughly if the genus plus the number of zeros is less than
seven, while the sum varies if this sum is greater than seven. 
\par
In this paper we show that a more precise version of this numerical observation 
indeed is true. More precisely, we treat the moduli space of genera less than or
equal to five. For each of its strata -- with three spin-related exceptions --
we either exhibit an example showing that the sum is varying -- the easy part 
-- or prove that the sum is non-varying.
The latter will be achieved by showing empty intersection of \teichm
curves with various geometrically defined divisors on moduli spaces of curves.
We remark that each stratum requires its own choice of divisor
and its individual proof of disjointness, with varying complexity
of the argument. In complement to our low genus results we mention
a theorem of \cite{ekz} that shows that for all hyperelliptic loci
the sum of Lyapunov exponents is non-varying.
\par
We now give the precise statement of what emerged out of the observation by Kontsevich and Zorich.
Let $(m_1,\ldots,m_k)$ be a partition of $2g-2$. 
Denote by $\omoduli(m_1,\ldots,m_k)$ the stratum parameterizing genus $g$ Riemann surfaces 
with Abelian differentials that have $k$ distinct zeros of order $m_1,\ldots,m_k$. We say that the
sum of Lyapunov exponents 
is {\em non-varying} in (a connected
component of) a stratum $\omoduli(m_1,\ldots,m_k)$, if for all \teichm curves 
generated by a flat surface in $\omoduli(m_1,\ldots,m_k)$ its sum
of Lyapunov exponents equals the sum for the finite invariant measure supported on
(the area one hypersurface of) the whole stratum. 
\par
\begin{theorem} \label{thm:g3main}
For all strata in genus $g=3$ 
but the principal stratum the sum of Lyapunov exponents is
non-varying.
\par
For the principal stratum, the sum of Lyapunov exponents is bounded above
by $2$. This bound can be attained for \teichm curves in the
hyperelliptic locus, e.g.,\ for \teichm curves that are unramified double
covers of genus two curves and also for \teichm curves that do not
lie in the hyperelliptic locus.
\end{theorem}
\par
\begin{theorem} \label{thm:g4main}
For the strata with signature $(6)^\even$, $(6)^\odd$, $(5,1)$,
$(3,3)$, $(3,2,1)$ and $(2,2,2)^\odd$ as well as for the hyperelliptic strata 
in genus $g=4$ the sum of Lyapunov exponents is
non-varying. 
\par
For all the remaining strata, except maybe $(4,2)^\odd$ and $(4,2)^\even$, the sum of Lyapunov exponents is varying and bounded
above by $5/2$.
\end{theorem}
\par
We give more precise upper bounds for the sum stratum by stratum in the
text. We remark that e.g.\ for $\omoduli[4](4,1,1)$ the sharp upper bound
is $23/10$, which is attained for hyperelliptic curves, whereas for all
non-hyperelliptic curves in this stratum the sum of Lyapunov exponents
is bounded above by $21/10$. This special role of the hyperelliptic
locus is visible throughout the paper.
\par
For $g=5$, since there are quite a lot of strata, we will not give a full discussion of 
upper bounds for varying sums, but restrict to the cases where the sum is non-varying.
\par
\begin{theorem} \label{thm:g5main}
For the strata with signature $(8)^\even$, $(8)^\odd$ and $(5,3)$ 
as well as for the hyperelliptic strata in genus $g=5$ the sum of Lyapunov exponents is non-varying. 
\par
For all the other strata, except maybe $(6,2)^\odd$, the sum of Lyapunov exponents is varying. 
\end{theorem}
\par
We also expect the three unconfirmed cases $(4,2)^{\even}$, $(4,2)^{\odd}$ and $(6,2)^{\odd}$
to be non-varying\footnote{Recently Yu and Zuo \cite{yuzuo} confirmed the three cases using filtration of the Hodge bundle. See also \cite[Theorem A.9]{lyapquad} for a detailed explanation.}, but a proof most likely requires a good understanding of the moduli space of spin
curves, on which much less is known than on the moduli space of curves.
\par
The above theorems seem to be the end of this non-varying phenomenon. We cannot claim that
there is not a single further stratum of genus greater than five and not
hyperelliptic, where the sum is non-varying. But while the sum in strata with
a single zero is always non-varying for $g \leq 5$, the sum does {\em vary}
in both non-hyperelliptic components of the stratum $\omoduli[6](10)$,
as we show in Proposition~\ref{prop:6vary}. 
\par
As mentioned above, by \cite{ekz} for hyperelliptic strata in any genus the sum of Lyapunov exponents is non-varying. This has significance not only in dynamics, but also 
in the study of birational geometry of moduli spaces. In Theorem~\ref{thm:extremal} we mention one application to the extremality of certain divisor 
classes on the moduli space of pointed curves, which answers a question posed by Harris \cite[p. 413]{harrisfamily}, 
\cite[Problem (6.34)]{harrismorrison}. 
\par 
We now describe our strategy. One can associate three quantities 'slope', 'Siegel-Veech constant' 
and 'the sum of Lyapunov exponents' to a \teichm curve. Any one of the three determines the other two. 
Hence it suffices to verify the non-varying property for slopes. 
To do this, we exhibit a geometrically defined divisor on the moduli space 
of curves and show that \teichm curves in a stratum do not
intersect this divisor. It implies that those \teichm curves
have the same slope as that of the divisor. 
\par
The slope of the divisors, more generally their divisor classes in the Picard
group of the moduli space, can be retrieved from the literature in
most cases we need. In the remaining cases, we apply the standard
procedure using test curves to calculate the divisor class. 
\par
Frequently, we also need to consider the moduli 
space of curves with marked points or spin structures, but the basic idea
remains the same. For upper bounds of sums of Lyapunov exponents, they
follow from the non-negative intersection property of \teichm curves with
various divisors on moduli spaces. 
\par
Technically, some of the complications arise from the fact, that the disjointness of a \teichm curve with a divisor is relatively
easy to check in the interior of the moduli space, but
requires extra care when dealing with stable nodal curves in the boundary.
\par
In a sequel paper \cite{lyapquad} we consider \teichm curves generated by quadratic differentials and 
verify many non-varying strata of quadratic differentials in low genus. 
These results immediately trigger a number of questions. Just to mention the most obvious ones. What about measures supported on manifolds of intermediate dimension? What about the value
distribution for the sums in a stratum where the sum is varying? 
We hope to treat these questions in the future.
\par
This paper is organized as follows. In Sections~\ref{sec:backmoduli}, \ref{sec:Teichgen} and \ref{sec:backLyap} we give 
a background introduction to moduli spaces and their divisors, as well as to \teichm curves and Lyapunov exponents. In particular, in Section~\ref{sec:propTeich} we study the properties of \teichm curves that are needed in the proof and in Section~\ref{sec:svsl} we describe
the upshot of our strategy. Our main results for $g=3$, $g=4$ and $g=5$ are 
proved in Sections~\ref{sec:g=3}, \ref{sec:g=4} and \ref{sec:g=5}, respectively. Finally in Section~\ref{sec:hyp} we discuss 
an application of the \teichm curves in the hyperelliptic strata to the geometry of moduli spaces of pointed curves. 
\par 
{\bf Acknowledgements.} This work was initiated when the first author visited the 
Hausdorff Research Institute for Mathematics in Summer 2010. Both authors thank HIM for 
hospitality. The evaluation of Lyapunov exponents were performed with the help of
computer programs of Anton Zorich and Vincent Delecroix. The authors are grateful to
Anton and Vincent for sharing the programs and the data. The first author also 
would like to thank Izzet Coskun for stimulating discussions on the 
geometry of canonical curves and thank Alex Eskin for leading him into the beautiful subject
of Teichm\"uller curves. Finally the authors want to thank the anonymous referees for a number of 
suggestions which helped improve the exposition of the paper. 

\section{Background on moduli spaces} \label{sec:backmoduli}

\subsection{Strata of $\omoduli$ and hyperelliptic loci}
Let $\omoduli$ denote the vector bundle of holomorphic one-forms
over the moduli space $\moduli$ of genus $g$ curves minus the zero section
and let $\pomoduli$ denote the associated projective bundle. The spaces
$\omoduli$ and $\pomoduli$ are stratified
according to the zeros of one-forms. For $m_i\geq 1$ and $\sum_{i=1}^k m_i = 2g-2$, 
let $\omoduli(m_1,\ldots,m_k)$ denote the stratum parameterizing one-forms that
have $k$ distinct zeros of order $m_1,\ldots, m_k$. 
\par
Denote by $\barmoduli$ the Deligne-Mumford compactification
of $\moduli$. The boundary of $\barmoduli$ parameterizes \emph{stable nodal curves}, where the stability means 
the dualizing sheaf of the curve is ample, or equivalently, the normalization of any rational component needs to possess at least three special points coming from the inverse images of the nodes. The bundle of holomorphic one-forms extends
over $\barmoduli$, parameterizing {\em stable one-forms}
or equivalently sections of the dualizing sheaf.
We denote the total space of this extension by $\obarmoduli$.
\par
Points in $\omoduli$, called {\em flat surfaces}, are usually
written as $(X,\omega)$ for a one-form $\omega$ on $X$.
For a stable curve $X$, denote the dualizing sheaf 
by $\omega_X$. We will stick to the notation that points in
$\obarmoduli$ are given by  a pair $(X,\omega)$ with $\omega \in H^0(X,\omega_X)$.
\par
For $d_i \geq -1$ and $\sum_{i=1}^s d_i = 4g-4$, let $\cQ(d_1,\ldots,d_s)$ 
denote the {\em moduli space of quadratic differentials} that have $s$ distinct zeros or poles of order 
$d_1,\ldots, d_s$. The condition $d_i\geq -1$ ensures that the quadratic differentials in $\cQ(d_1,\ldots,d_s)$ have at most simple poles. Namely, $\cQ(d_1,\ldots,d_s)$ parameterizes pairs $(X,q)$
of a Riemann surface $X$ and a meromorphic section $q$ of $\omega_X^{\otimes 2}$
with the prescribed type of zeros and poles. Pairs  $(X,q)$ are called {\em half-translation
surfaces}. They appear occasionally to provide examples via the following construction.
\par
If the quadratic differential is not a global square of a one-form, 
there is a natural double covering $\pi: Y \to X$
such that $\pi^* q = \omega^2$. This covering is ramified precisely
at the zeros of odd order of $q$ and at its poles. It gives a
map
$$\phi: \cQ(d_1,\ldots,d_s) \to \omoduli(m_1,\ldots, m_k), $$
where the signature $(m_1,\ldots, m_k)$ is determined by the ramification
type. Indeed $\phi$ is an immersion (see \cite[Lemma 1]{kz03}).
\par
There are two cases where the domain and the range of the map $\phi$ have the same dimension: 
$$ \cQ(-1^{2g+1}, 2g-3) \to \omoduli(2g-2), $$
$$ \cQ(-1^{2g+2}, 2g-2) \to \omoduli(g-1,g-1), $$
see \cite[p. 637]{kz03}. 
In both cases we call the image a {\em component of hyperelliptic flat surfaces}
of the corresponding stratum of Abelian differentials. Note that for both cases the domain of $\phi$ 
parameterizes genus zero curves. More generally, if the domain of $\phi$ parameterizes
genus zero curves, we call the image a {\em locus of hyperelliptic flat surfaces}
in the corresponding stratum. These loci are often called hyperelliptic loci, 
e.g.\ in \cite{kz03} and \cite{ekz}. We prefer to reserve {\em hyperelliptic locus} for the
subset of $\moduli$ (or its closure in $\barmoduli$, see also Section~\ref{sec:specialdivisors}) 
parameterizing hyperelliptic curves 
and thus specify with 'flat surfaces' if we speak of subsets of $\omoduli$. 
\par


\subsection{Spin structures and connected components
of strata} \label{sec:spinmod}

A {\em spin structure} (or {\em theta characteristic}) 
on a smooth curve $X$ is a line bundle $\cL$
whose square is the canonical bundle, i.e. $\cL^{\otimes 2} \sim K_X$.
The {\em parity of a spin structure} is given by $h^0(\cL) \mod 2$. This
parity is well-known to be deformation invariant. There is
a notion of spin structure on a stable curve, extending the
smooth case (see \cite{Corn89}, also recalled in \cite[Section 1]{FV}). We only
need the following consequence.
The {\em moduli space of spin curves} $\Sg$ parameterizes pairs 
$(X,\eta)$, where $\eta$ is a theta characteristic of $X$.
It has two components $\Sgo$ and $\Sge$ distinguished by
the parity of the spin structure. The spin structures on
stable curves are defined such that the morphisms 
$\pi: \Sgo\rightarrow \Mg$ and $\pi: \Sge \rightarrow \Mg$ are 
finite of degree $2^{g-1}(2^g-1)$ and $2^{g-1}(2^g+1)$, respectively, cf. loc.\ cit.
\par
Recall the classification of connected components 
of strata in $\omoduli$ by Kontsevich and Zorich \cite[Theorem 1 on p. 639]{kz03}.
\par
\begin{theorem}[\cite{kz03}]
The strata of $\omoduli$ have up to three
connected components, distinguished by the parity of the spin structure
and by being hyperelliptic or not. For $g \geq 4$, the strata
$\omoduli[g](2g-2)$ and $\omoduli[g](2k,2k)$ with an integer $k = (g-1)/2$
have three components,
the component of  hyperelliptic flat surfaces and two components with odd or even
parity of the spin structure but not consisting exclusively
of hyperelliptic curves. 
\par
The stratum $\omoduli[3](4)$ has two components, $\omoduli[3](4)^\hyp$ and $\omoduli[3](4)^\odd.$
The stratum $\omoduli[3](2,2)$ also has two components, $\omoduli[3](2,2)^\hyp$ and $\omoduli[3](2,2)^\odd.$
\par
Each stratum $\omoduli[g](2k_1,\ldots,2k_r)$ for $r \geq 3$ or $r=2$ and $k_1 \neq (g-1)/2$
has two components determined by even and odd spin structures.
\par
Each stratum $\omoduli[g](2k-1,2k-1)$ for $k\geq 2$ has two components, the
 component of hyperelliptic flat surfaces $\omoduli[g](2k-1,2k-1)^\hyp$ and the other component $\omoduli[g](2k-1,2k-1)^\nh$.\par 
In all the other cases, the stratum is connected.
\end{theorem}
\par
Consider the partition $(2, \ldots, 2)$. For 
$(X,\omega) \in \omoduli(2,\ldots, 2)^{\odd}$ with 
div($\omega) = 2\sum_{i=1}^{g-1}p_i$, the line bundle $\eta = \OO_X(\sum_{i=1}^{g-1}p_i)$ 
is an odd theta characteristic. Therefore, we have a natural morphism 
$$f:  \omoduli(2,\ldots, 2)^{\odd}\rightarrow \Sgo. $$
Note that $f$ contracts the locus where $h^{0}(\eta) > 1$. 
Similarly one can define such a morphism for even spin structures. 
\par

\subsection{Picard groups of moduli spaces}\label{sec:picmoduli}

Let $\moduli[{g,n}]$ be the moduli space (treated as a stack instead of the course moduli scheme) of genus $g$ curves with
$n$ ordered marked points and let $\moduli[{g,[n]}]$ be the moduli space of genus $g$ curves with
$n$ unordered marked points. We write
$\Pic(\cdot)$ for the rational Picard group $\Pic_{\fun}(\cdot)_\QQ$ of a moduli stack (see e.g. \cite[Section 3.D]{harrismorrison} for more details).
\par
We fix some standard notation for elements in the Picard group. Let $\lambda$
denote the first Chern class of the Hodge bundle. Let $\delta_i$, $i=1,\ldots, \lfloor g/2 \rfloor$
be the boundary divisor of $\barmoduli$ whose generic element is a smooth curve of genus $i$
joined at a node to a smooth curve of genus $g-i$. The generic element of the
boundary divisor $\delta_0$ is an irreducible nodal curve of geometric genus $g-1$. In the literature sometimes 
$\delta_0$ is denoted by $\delta_{\irr}$. We write $\delta$ for the total boundary class. 
\par
For moduli spaces with marked points we denote by $\omega_\rel$ the
relative dualizing sheaf of $\barmoduli[{g,1}] \to \barmoduli$ and $\omega_{i,\rel}$ its pullback to $\barmoduli[{g,n}]$ via
the map forgetting all but the $i$-th marked point. For a set $S \subset
\{1,\ldots,n\}$ we let $\delta_{i;S}$ denote the boundary divisor whose generic 
element is a smooth curve of genus $i$ joined at a node to a smooth curve of genus $g-i$
and the sections in $S$ lying on the first component.
\par
\begin{theorem}
The rational Picard group of $\barmoduli$ is generated by $\lambda$ and the
boundary classes $\delta_i$, $i=0,\ldots, \lfloor g/2 \rfloor$. 
\par
More generally, the rational Picard group of $\barmoduli[{g,n}]$ is generated by $\lambda$,
$\omega_{i,\rel}$, $i=1,\ldots,n$, by $\delta_0$ and by 
$\delta_{i;S}$,  $i=0,\ldots, \lfloor g/2 \rfloor$, 
where $|S| > 1$ if $i=0$ and $1 \in S$ if $i= g/2$.
\end{theorem}
The above theorem essentially follows from Harer's result $H^0(\barmoduli, \QQ) \cong \QQ[\lambda]$ \cite{harer}. 
The reader may also refer to \cite{mumfordstability} for a comparison between the rational Picard group of the coarse moduli scheme and of the moduli stack, as well as \cite{AC} for the Picard group with integral coefficients. 
\par
Alternatively, define $\psi_i \in \Pic(\barmoduli[{g,n}])$ to be the class
with value $-\pi_*(\sigma_i^2)$ on any family of stable genus $g$ curves $\pi: \cX \to C$ with section $\sigma_i$ corresponding to the $i$-th marked point. By induction on $n$, we have the relation (see e.g. \cite[p. 161]{AC} and \cite[p. 108]{Logan})
$$ \omega_{i,\rel} = \psi_i - \sum_{i\in S} \delta_{0; S}.  $$ 
Consequently, a generating set of $\Pic(\barmoduli[{g,n}])$ can also be formed
by the $\psi_i$, $\lambda$ and boundary classes.
\par
For a divisor class $D = a \lambda -  \sum_{i=0}^{\lfloor g/2 \rfloor} b_i \delta_i$ in $\Pic(\barmoduli)$, 
define its {\em slope} to be
$$ s(D) = \frac{a}{b_0}.$$ 
For our purpose the higher boundary divisors need not to be considered, as \teichm curves generated by Abelian differentials
do not intersect $\delta_i$ for $i > 0$ (see Corollary~\ref{cor:zerosatinf}). 

\subsection{Linear series on curves}
Many divisors on moduli spaces of curves are related to the geometry of \emph{linear series}. Here we review some basic 
properties of linear series on curves (see \cite{ACGH} for a comprehensive introduction). 

Let $X$ be a genus $g$ curve
and $\cL$ a line bundle of degree $d$ on $X$. Denote by $|\cL|$ the linear system parameterizing sections of $\cL$ mod scalars, i.e. 
$$|\cL| = \{ \mbox{div}(s)\ |\ s\in H^0(\cL) \}.$$ 
If $h^0(\cL) = n$, then $|\cL| \cong \PP^{n-1}$. For a  
(projective) $r$-dimensional linear subspace $V$ of $|\cL|$, call $(\cL, V)$ a linear series $g^r_d$. If $\cL\sim \OO_X(D)$ for a divisor $D$ on $X$, we also denote by $|\OO_X(D)|$ or simply 
by $|D|$ the linear system.  

If all divisors parameterized in a linear series contain a common point $p$, then $p$ is called a \emph{base point}. Otherwise, this linear series is called \emph{base-point-free}. A base-point-free $g^r_d$ induces a morphism 
$X\rightarrow \PP^r$. The divisors in this $g^r_d$ correspond to (the pullback of) hyperplane sections 
of the image curve. For instance, a hyperelliptic curve admits a $g^1_2$, i.e. a double cover of $\PP^1$. The following fact will be used frequently when we prove the disjointness of Teichm\"{u}ller curves with a geometrically defined divisor. 

\begin{prop}\label{bpf}
A point $p$ is not a base point of a linear system $|\cL|$ if and only if 
$h^0(\cL) -1 =  h^0(\cL(-p))$, 
where $\cL(-p) = \cL \otimes \OO_X(-p)$.  
\end{prop}

\begin{proof}
By the exact sequence 
$$ 0 \to \cL(-p) \to \cL \to \OO_p \to 0, $$
we know $h^0(\cL(-p))$ is either equal to $h^0(\cL)$ or $h^0(\cL) - 1$. The former happens if and only if every section of $|\cL|$ vanishes at $p$, in other words, if and only if $p$ is a base point of $|\cL|$. 
\end{proof}

The canonical linear system is a $g^{g-1}_{2g-2}$, which induces an embedding to $\PP^{g-1}$ for a non-hyperelliptic curve. The image of this embedding is called a \emph{canonical curve}. Let $D$ be an effective divisor of degree $d$ on $X$. Denote by $\overline{\phi_K(D)}$ the linear subspace in $\PP^{g-1}$ spanned by the images of points in $D$ under the canonical map $\phi_K$. The following geometric version of the Riemann-Roch theorem is useful for the study of canonical curves (see \cite[p. 12]{ACGH} for more details). 

\begin{theorem}[Geometric Riemann-Roch]
In the above setting, we have 
$$\mbox{{\rm dim}} \ |D| = d - 1 - \mbox{{\rm dim}}\ \overline{\phi_K(D)}.$$ 
\end{theorem}

We will focus on the geometry of canonical curves of low genus. Curves of genus~$2$ are always hyperelliptic. For non-hyperelliptic curves of genus $3$, their canonical images correspond to plane quartics. 

For $g = 4$, a non-hyperelliptic canonical curve $X$ in $\PP^3$ is a complete intersection cut out by a quadric and a cubic. Any divisor $D = p+q+r$ in a $g^1_3$ of $X$ spans a line in $\PP^3$, by Geometric Riemann-Roch. This line intersects $X$ at $p,q,r$, hence it is contained in the quadric by B\'{e}zout. If the quadric is smooth, it is isomorphic to $\PP^1\times \PP^1$. It has two families of lines, called two \emph{rulings}. Any line in a ruling intersects $X$ at three points (with multiplicity), hence $X$ has two different linear systems $g^1_3$ corresponding to the two rulings. If the quadric is singular, then it is a quadric cone with a unique ruling, hence $X$ has a unique $g^1_3$. 
 
For $g=5$, a general canonical curve is cut out by three quadric hypersurfaces in $\PP^4$ and it does not have any $g^1_3$. On the other hand, a genus $5$ curve with a $g^1_3$, i.e. a \emph{trigonal curve}, has canonical image contained in a \emph{cubic scroll surface}, which is a ruled surface of degree $3$ in 
$\PP^4$. By Geometric Riemann-Roch, divisors in the $g^1_3$ span rulings that sweep out the surface (see e.g. \cite[Section 2.10]{reid}).  

Recall that on a nodal curve $X$, Serre duality and Riemann-Roch hold
with the dualizing sheaf $\omega_X$ in place of the canonical bundle (see e.g. \cite[Section 3.A]{harrismorrison} for more details). 
We also need the following generalized Clifford's theorem for Deligne-Mumford stable curves (see e.g. \cite[p. 107]{ACGH} for the case 
of smooth curves and \cite[Theorems 3.3, 4.11]{caplinear} for the remaining cases). Following \cite[Section~2.1]{caplinear} 
let $\delta_Z = Z \cdot Z^c$ be the number of intersection points of $Z$ with its complement
and $w_Z = 2g(Z)-2 + \delta_Z$. A divisor $D$ of degree $d$ on a stable curve $X$ is {\em balanced}, if for every irreducible component $Z \subset X$ we have 
$$d \frac{w_Z}{2g-2} - \frac{\delta_Z}{2} \leq \deg(D|_Z) \leq d \frac{w_Z}{2g-2} - \frac{\delta_Z}{2}.$$
\par
\begin{theorem} [Clifford's theorem]\label{thm:clifford}
Let $X$ be a stable curve and $D$ an effective divisor on $X$ with $\deg(D) \leq 2g-1$.
Then we have $$h^0(\cO_X(D))-1 \leq \deg(D)/2 $$
if one of the following conditions holds: (i) $X$ is smooth; (ii) $X$ has at most two components
and $D$ is balanced; (iii) $X$ does not have separating
nodes, $\deg(D) \leq 4$ and $D$ is balanced.
\end{theorem}

Finally we need to consider the canonical system of a stable curve associated to its dualizing sheaf. This will help us discuss the boundary of \teichm curves. Recall the \emph{dual graph} of a nodal curve whose vertices correspond to its irreducible components and edges correspond to intersections of these components. A graph is called \emph{$n$-connected} if one has to remove 
at least $n$ edges to disconnect the graph. The following fact characterizes canonical maps of stable curves based on the type of their dual graphs 
(see \cite[Proposition 2.3]{Hassett}). 

\begin{prop}
\label{prop:dualgraph}
Let $X$ be a stable curve of genus $\geq 2$. Then the canonical linear system 
$|\omega_X|$ is base point free (resp. very ample) if and only if the dual graph of $X$ 
is two-connected (resp. three-connected and $X$ is not in the closure of the locus of hyperelliptic curves). 
\end{prop}

\subsection{Special divisors on moduli spaces} \label{sec:specialdivisors}

In the application for \teichm curves generated by flat surfaces
we do not care
about the coefficients of $\delta_i$ for $i \geq 1$ in the 
divisor classes in $\Pic(\barmoduli)$, since
\teichm curves do not intersect those components
(see Corollary~\ref{cor:zerosatinf}). As shorthand, we use $\delta_\other$ to 
denote some \emph{linear combination of  $\delta_i$ for $i \geq 1$}.
Similarly, in $\barmoduli[{g,n}]$ we use $\delta_\other$ to 
denote some \emph{linear combination of all boundary divisors but  $\delta_0$}. 
By the same reason we do not distinguish between $\omega_{i,\rel}$ and $\psi_i$ for a divisor class, since 
they only differ by boundary classes in $\delta_\other$. 
\subsubsection*{The hyperelliptic locus in $\barmoduli[3]$} 

Denote by $H \subset \barmoduli$ the
closure of locus of genus $g$ hyperelliptic curves. 
We call $H$ the hyperelliptic locus in $\barmoduli$. 
Note that $H$ is a divisor if and only if $g=3$. A stable curve $X$ lies
in the boundary of $H$ if there is an {\em admissible cover} of degree two
$\tilde{X} \to \PP^1$, for some nodal curve $\tilde{X}$ whose stabilization is $X$. We refer to \cite[Section 3.G]{harrismorrison} for an excellent introduction to admissible covers.
\par
The class of the hyperelliptic locus $H \subset \barmoduli[3]$ calculated e.g.\ in 
\cite[p. 188]{harrismorrison} is given as follows:
\begin{equation} \label{eq:classofH}
H = 9\lambda - \delta_0 - 3\delta_1, 
\end{equation}
hence it has slope $s(H) = 9$.  

\subsubsection*{Divisors of Weierstrass points}

Let $W\subset \barmoduli[{g,1}]$ be the divisor parameterizing a curve 
with a Weierstrass point. In \cite[(2.0.12) on p. 328]{Cu}, the class 
of $W$ was calculated for all $g$, which specializes as follows: 
\begin{equation} \label{eq:classofW}
W = 6\omega_\rel - \lambda - \delta_\other \quad  \text{for} \quad g=3. 
\end{equation}
\par

\subsubsection*{The theta-null divisor}

Consider the divisor $\Theta\subset \barmoduli[{g,1}]$ parameterizing $(X, p)$ such that 
$X$ admits an odd theta characteristic whose support contains $p$.  
The class of $\Theta$ was calculated in \cite[Theorem 0.2]{F2}, which specializes as follows:
\begin{equation}\label{eq:thetanull}
\Theta = 30\lambda + 60\omega_\rel - 4\delta_0 - \delta_\other \quad \text{for} \quad g=4. 
\end{equation}

\subsubsection*{The Brill-Noether divisors}

The Brill-Noether locus $BN^r_d$ in $\barmoduli$ parameterizes
curves $X$ that possesses a $g^r_d$. If the Brill-Noether number
$$\rho(g,r,d) = g - (r+1)(g-d+r) = -1,$$ then $BN^r_d$ is indeed a divisor. 
We remark that nowadays $\mathcal{M}^r_{g,d}$ is more commonly used to denote 
the Brill-Noether divisors, but we decide to reserve $\mathcal{M}$ for the moduli space only.     
\par
There are pointed versions of this divisor. Let $\ulw = (w_1,\ldots,w_n)$ be a tuple of
integers. Let $BN^r_{d,\ulw}$ be the locus in $\barmoduli[{g,n}]$
of pointed curves $(X,p_1,\ldots,p_n)$ with a line bundle $\cL$ of degree $d$ such that
$\cL$ admits a $g^r_d$ and $h^0(\cL(-\sum w_i p_i)) \geq r$. This Brill-Noether locus is a divisor, if
the generalized Brill-Noether number
$$ \rho(g,r,d,\ulw) = g-(r+1)(g-d+r)-r(|\ulw|-1) = -1.$$
The hyperelliptic divisor and
the Weierstrass divisor could also be interpreted as Brill-Noether divisors, but
we stick to the traditional notation for them.
\par
The class of these pointed divisors has been calculated in many special cases, in
particular in \cite{Logan} and later in \cite{F1}. 
We collect the results that are needed here.
\par
The class of the classical Brill-Noether divisor for $r=1$ was calculated in \cite[p. 24]{harrismumford}, 
in particular
\begin{equation}
\label{BN}
 BN^1_{3} = 8\lambda - \delta_0  - \delta_\other \quad \text{for} \quad g=5.
\end{equation}
If $|\ulw|=g=d$ and $r=1$ the class of the Brill-Noether divisor was 
calculated in \cite[Theorem 5.4]{Logan}. It has class
$$ BN^1_{g,\ulw} = -\lambda + \sum_{i=1}^k \frac{w_i(w_i+1)}{2}
\omega_{i,\rel} - \delta_\other.$$  In particular for $\ulw = (1,2)$, it specializes as follows: 
\begin{equation} \label{eq:BN(3,1,2)}
BN^1_{3,(1,2)} = -\lambda + \omega_{1,\rel} + 3\omega_{2,\rel} - \delta_\other \quad \text{for} \quad g=3,
\end{equation}
For $\ulw = (1,1,2)$, it specializes as follows:  
\begin{equation}\label{BN(1,1,2)}
BN^1_{4,(1,1,2)} = -\lambda + \omega_{1,\rel} + \omega_{2,\rel} + 3\omega_{3,\rel} - \delta_\other \quad \text{for} \quad g=4. 
\end{equation}
\par
If $r=1$ and $\ulw = (2)$, the class of the divisor was also calculated in \cite{Logan}. It specializes to
\begin{equation}
\label{BNp}
 BN^1_{3,(2)} = 4\omega_\rel + 8\lambda - \delta_0 - \delta_\other \quad \text{for} \quad g=4.
\end{equation}
\par
If all $w_i=1$ and $n=r+1$ the Brill-Noether divisor specializes to the divisor $\Lin$ 
calculated in \cite[Section 4.2]{F1}. In particular \cite[Theorem 4.6]{F1} gives
\begin{equation}
\label{eq:Lin}
 \Lin^1_{3} = BN_{3,(1,1)}^1 = -\omega_{1,\rel} -\omega_{2,\rel}  + 8\lambda - \delta_0 -\delta_\other
\quad \text{for} \quad g=4.
\end{equation}
\par
Generalizing the calculation of Logan for $r=1$ and $n=1$ to arbitrary weight $w_1$, 
one obtains the divisor called $\Nfold^1_{d}(1)$ in the proof of \cite[Theorem 4.9]{F1}. From 
the proof one deduces 
\begin{equation}
\label{eq:Nfold}
 \Nfold^1_{4}(1) = BN^1_{4,(3)} = 7\lambda + 15\omega_\rel - \delta_0 - \delta_\other
\quad \text{for} \quad g=5.
\end{equation}
\par
$\Nfold(1)$ is a degeneration of the divisor $\Nfold$ in \cite{F1}. A partial degeneration
is $\Nfold(2) = BN^1_{4,(1,2)}$ in $\barmoduli[{5,2}]$. It has 
class 
\begin{equation}
\label{eq:Nfold2}
\Nfold^1_4(2) = BN^1_{4,(1,2)} =  7\lambda + 7 \omega_{1,\rel} + 2 \omega_{2,\rel} - \delta_{0} - \delta_\other
\quad \text{for} \quad g=5. 
\end{equation}
Since this divisor class was not explicitly written out in \cite{F1}, below we give a proof. 
\par
\begin{proof}[Proof of Equation~\eqref{eq:Nfold2}]
Using the same logic in the proof of \cite[Theorems 4.6, 4.9]{F1}, $\lambda, \delta_{0}, \psi_1$ 
have non-varying coefficients in $\Nfold^1_4$, which is $BN^1_{4,(1,1,1)}$ in our notation, 
and in $\Nfold^1_4(2)$. Hence we have 
$$ \Nfold^1_4(2) = 7\lambda -\delta_0 + 2\psi_1 + c\psi_2 - e\delta_{0;\{1,2\}} - \delta_\other.$$
We have to take $\delta_{0;\{1,2\}}$ into account, because the test curves used below intersect $\delta_{0;\{1,2\}}$. 
Let $X$ be a general curve of genus five. Take a fixed general point $x_2$ on $X$ and 
move another point $x_1$ along $C$. Call this family $B_1$. 
We have 
$$B_1 \cdot \lambda = 0, \quad B_1\cdot\delta_0 = 0, \quad B_1\cdot \delta_{0;\{1,2\}} = 1, $$
$$ B_1\cdot \psi_1 =  9, \quad B_1\cdot \psi_2 = 1.$$
The intersection number $B_1\cdot \Nfold^1_4(2)$ can be calculated using \cite[Proposition 3.4]{Logan} by setting 
$a_1 = 2, a_2 = 1, g=5, h = 1$, and it equals $10$. Note that Logan counts the number of pairs 
($p_2, q_1$), which equals $5$, but for our purpose $x_1$ can be either $p_2$ or $q_1$, so we double 
the counting. We thus obtain a relation 
$$ c - e + 8 = 0. $$
Now fix a general point $x_1$ and move another point $x_2$ along  $X$. Call this family $B_2$. 
We have 
$$B_2 \cdot \lambda = 0, \quad B_2\cdot \delta_{0;\{1,2\}} = 1, $$ 
$$ B_2\cdot \psi_1 =  1, \quad B_2\cdot \psi_2 = 9. $$
The intersection number $B_2\cdot \Nfold(2)$ can also be calculated using \cite[Proposition 3.4]{Logan} by setting 
$a_1 = 1, a_2 = 2, g=5, h = 1$, and 
it equals $50$. This equals Logan's counting, since in the pair $(p_2, q_1)$ now $p_2$ has weight $2$, which distinguishes it from $q_1$. 
We then obtain another relation
$$9c - e - 48 = 0. $$  
Combining the two relations we conclude that $c= 7, e = 15$, which completes the proof.
\end{proof}

\subsubsection*{Gieseker-Petri divisors}

Consider a linear series $(\cL,V) \in G^r_d(X)$ for a linear subspace $V\subset H^0(\cL)$ of dimension $r+1$, 
$\deg(\cL) = d$ and the multiplication map
$$\mu: V \otimes H^0(\omega_X \otimes \cL^{-1}) \to H^0(\omega_X).$$
Define the Gieseker-Petri locus 
$$GP^r_{g,d} = \{[X] \in \barmoduli[g], \exists \,\text{base-point-free}\,
(\cL, V) \in G^r_d(X) \,
\text{such that}\, \mu \, \text{is not injective} \}.$$
The divisor class of the Gieseker-Petri locus in the case $r=1$ was calculated in \cite[Theorem 2]{eisenbudharris}. It specializes to 
 \begin{equation}\label{eq:GP}
 GP = 17 \lambda - 2 \delta_0 + \delta_\other \quad \text{for} \quad g=4.
 \end{equation}
Alternatively, one can describe $GP$ in $\moduli[4]$ as follows. The
canonical image of a genus $4$ non-hyperelliptic curve is contained in a quadric surface
in $\PP^3$. Then $GP$ is the closure of the locus where this quadric
is singular (see e.g. \cite[p. 196]{ACGH}).

\section{\teichm curves and their boundary points} \label{sec:Teichgen}

We quickly recall the definition of Teichm\"uller curves and of square-tiled surfaces which serve as main examples. New results on the boundary behavior of \teichm curves needed later are collected in Section~\ref{sec:propTeich}.

\subsection{\teichm curves as fibered  surfaces}

A {\em \teichm curve} $C \to \moduli$ is an algebraic curve in the moduli
space of curves that is totally geodesic with respect to the \teichm metric.
There exists a finite unramified cover $B \to C$ such that the monodromies
around the 'punctures' $\overline{B} \setminus B$ are unipotent and such that the universal family
over some level covering of $\moduli$ pulls back to a family of curves
$f: \cX \to B$. We denote by $f: \bcX \to \bB$ a relatively minimal 
semistable model of a fibered  surface of fiber genus $g$ with smooth total 
space. Let $\Delta \subset \bB$ be the set of points with
singular fibers, hence $B = \bB \setminus \Delta$. See e.g.\ \cite{moeller06} for
more on this setup. By a further finite unramified covering (outside $\Delta$)
we may suppose that the zeros of $\omega$ on $X$ extend to 
sections $\sigma_i$ of $f$. We denote by $S_i \subset \bcX$ the images of these
sections. 
\par
\teichm curves arise as the $\SL_2(\RR)$-orbit of special flat surfaces or half-translation surfaces,
called {\em Veech surfaces}. We deal here with the first case only and denote by $(X,\omega)$ a {\em generating flat surface}, if its 
$\SL_2(\RR)$ orbit gives rise to a \teichm curve. 
\teichm curves come with a uniformization $C = \mathbb H/\SL(X,\omega)$,
where $\SL(X,\omega)$ is the {\em affine group} (or  {\em Veech group}) of the flat
surface $(X,\omega)$. Let $K = \QQ(\tr(\gamma), \gamma \in \SL(X,\omega))$
denote the {\em trace field} of the affine group and let $L/\QQ$ denote the
Galois closure of $K/\QQ$.
\par
The variation of Hodge structure (VHS) over a \teichm curve decomposes
into sub-VHS
$$ R^1 f_* \CC = (\oplus_{\sigma \in \Gal(L/\QQ)/\Gal(K/\QQ)} \LL^\sigma) \oplus \MM,$$
where $\LL$ is the VHS with the standard 'affine group' representation,
$\LL^\sigma$ are the Galois conjugates 
and $\MM$ is just some representation (\cite[Proposition 2.4]{moeller06}).
One of the purposes of our work is to shed some light on what possibilities
for the numerical data of $\MM$ can occur. 
\par

\subsection{Square-tiled surfaces} \label{sec:sqtiled}

A {\em square-tiled surface} is a flat surface $(X,\omega)$, where
$X$ is obtained as a covering of a torus ramified over one point only
and $\omega$ is the pullback of a holomorphic one-form on the torus.
It is well-known that in this case $\SL(X,\omega)$ is commensurable to
$\SL_2(\ZZ)$, hence $\LL$ has no Galois conjugates or equivalently, 
the rank of $\MM$ is $2g-2$.
\par
In order to specify a square-tiled surface covered by $d$ squares, 
it suffices to specify the monodromy of the covering. Take a standard torus $E$ by 
identifying via affine translation the two pairs of parallel edges of the unit square
$[0,1] \times [0, i]$. Consider the closed, oriented paths $u = [0,1]$ on the horizontal axis and 
$r = [0, i]$ on the vertical axis. The indices $u$ and $r$ correspond to 'up' and 'right', 
respectively. Note that $u$ and $r$ form a basis of $\pi_1(E, b)$, where $b$ is a base point in $E$.  
Going along $u$ and $r$ induces two permutations $(\pi_u,\pi_r)$ on the $d$ sheets of a degree $d$ cover of $E$.  
Hence $\pi_u, \pi_r$ can be regarded as elements in the symmetric group $S_d$. Conversely, given such a pair 
$(\pi_u, \pi_r)$, one can construct a degree $d$ cover of $E$ (possibly disconnected) ramified over one point only. The domain of the covering is connected if and only if the subgroup in $S_d$ generated by $\pi_u, \pi_r$ acts transitively on the $d$ letters. Moreover, the ramification profile
over $b$ is determined by the commutator $\pi_{u}^{-1}\pi_{r}^{-1} \pi_{u}\pi_{r}$.  
\par
The surface in Figure~\ref{fig:square-tiled} corresponds to a degree 5, genus 2, connected cover of the standard torus: 
\begin{figure}[H]
    \centering
    \includegraphics[scale=0.5]{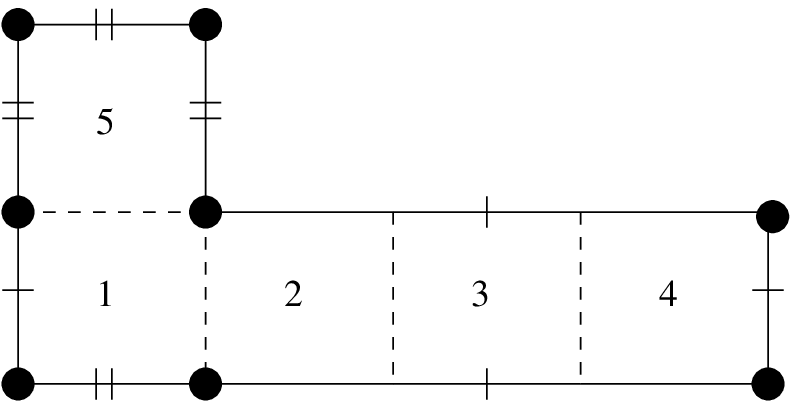}
    \caption{A square-tiled surface of degree $5$ and genus $2$}
    \label{fig:square-tiled}
   \end{figure}
It is easy to see that the monodromy permutations for this square-tiled surface are given by 
$(\pi_u = (15)(2)(3)(4), \pi_r = (1234)(5))$. Here a cycle $(a_1\ldots a_k)$ means 
the permutation sends $a_i$ to $a_{i+1}$ for $1\leq i \leq k-1$ and sends $a_k$ back to $a_1$.  
One can check that 
$\pi_{u}^{-1}\pi_{r}^{-1} \pi_{u}\pi_{r} = (154)(2)(3)$. Therefore, the corresponding covering has a unique 
ramification point marked by $\bullet$ with ramification order $2 = 3-1$ arising from the length-$3$ cycle, since 
locally the three sheets labeled by $1, 5, 4$ get permuted at that point. By Riemann-Hurwitz, the domain of the covering has genus equal to $2$. 
The pullback of $dz$ from $E$ is a one-form in the stratum $\Omega\mathcal{M}_2(2)$. 
\par
Based on the monodromy data, one can directly calculate the Siegel-Veech constant 
as well as the sum of Lyapunov exponents (introduced in the next section) for a Teichm\"uller curve 
generated by a square-tiled surface (see \cite{ekz}). Later on we will use square-tiled surfaces to produce examples of 
Teichm\"uller curves that have varying sums of Lyapunov exponents. 

\subsection{Properties of \teichm curves} \label{sec:propTeich} 

Here we collect the properties of the boundary points of \teichm curves that are needed in the proofs
in the subsequent sections. We will use $\bC$ to denote the \emph{closure} of a \teichm curve $C$ in the compactified moduli space.  
\par
Let $\mu$ be a partition of $2g-2$. If $\mu'$ is another partition and if it can be obtained from $\mu$ by successively 
combining two entries into one, we say that $\mu'$ is a degeneration of $\mu$. For instance, $(2,6)$ is a degeneration 
of $(1,1,3,3)$. Geometrically speaking, combining two entries $i, j$ corresponds to merging two zeros of order $i, j$ into a single zero of order $i+j$.  
\begin{prop} \label{prop:Teichbddisj}
Suppose $C$ is a \teichm curve generated by
a flat surface in $\omoduli[g](\mu)$ and let $\mu'$
be a degeneration of $\mu$. Then $\bC$ in $\pobarm(\mu)$ is disjoint from $\pobarm(\mu')$.
\end{prop}
\par
\begin{proof}
The claim is obvious over the interior of the moduli space. We only need to check
the disjointness over the boundary. The cusps of \teichm curves are obtained by applying
the \teichm geodesic flow ${\rm diag}(e^t,e^{-t})$
to the direction of the flat surface $(X,\omega)$ in which
$(X,\omega)$ decomposes completely into cylinders. The
stable surface at the cusp is obtained by 'squeezing' the core curves of 
these cylinders. This follows from the explicit description
in \cite{masur75}. Since the zeros of $\omega$ are located
away from the core curves of the cylinders, the claim follows.
\end{proof}
\par
For a nodal curve, a node is called \emph{separating} if removing it disconnects the curve. 
\begin{cor} \label{cor:zerosatinf}
The section $\omega$ of the canonical bundle of each smooth
fiber over a \teichm curve $C$ extends to a section $\omega_\infty$ of the dualizing sheaf 
for each singular fiber $X_\infty$ over the closure of a \teichm
curve. The signature of zeros of $\omega_\infty$ is the same
as that of $\omega$. Moreover, $X_\infty$ does not have separating nodes. In particular, $\bC$ does not intersect $\delta_i$ for $i > 0$. 
\end{cor}
\par
\begin{proof}
The first statement follows from the description in the preceding proof. 
The fact that $X_\infty$ does not have separating nodes is a consequence of the topological fact that
a core curve of a cylinder can never disconnect a flat surface. It implies that $\bC$ does not intersect the boundary divisors $\delta_i$ for $i>0$ on $\barmoduli$, because by definition
a curve parameterized in $\delta_i$ for $i>0$ possesses at least one separating node. 
\end{proof}
\par
\begin{cor} \label{cor:degfibers}
For \teichm curves generated by a flat surface in
$\omoduli[g](2g-2)$ the degenerate fibers are irreducible.
\par
For \teichm curves generated by a flat surface in
$\omoduli[g](k_1,k_2)$, with $k_1 \geq k_2$ both odd, the 
degenerate fibers are irreducible or consist of two components of genus $g_i$ for $i= 1, 2$ 
joined at $n$ nodes for an odd number $n$ such that $2g_i - 2 + n = k_i$.
\end{cor}
\begin{proof}
Let $X$ be a degenerate fiber and $Z$ a component of $X$. The dualizing sheaf of $X$ restricted to $Z$ has positive degree 
equal to $2g_{Z} - 2 + \delta_Z$, where $\delta_Z$ is the intersection number of $Z$ with its complement in $X$. For the case $\omoduli[g](2g-2)$, by Corollary~\ref{cor:zerosatinf} it implies that $X$ only has one component, hence it is irreducible. For the case 
$\omoduli[g](k_1,k_2)$, it implies that $X$ is either irreducible or has two components $Z_1, Z_2$. For the latter 
suppose $Z_i$ contains the $k_i$-fold zero. By assumption $2g_{i} - 2 + \delta_{Z_i} = k_i$ is odd, hence $\delta_{Z_1} = \delta_{Z_2} = n$ is also odd. 
\end{proof}
\par
\begin{prop} \label{prop:Teichhypdisj}
Let $C$ be a \teichm curve generated by
a flat surface $(X,\omega)$ in $\omoduli[g](\mu)$.
Suppose an irreducible degenerate fiber $X_\infty$ over a cusp
of $C$ is hyperelliptic. Then $X$ is hyperelliptic, hence
the whole \teichm curve lies in the locus of hyperelliptic
flat surfaces. 
\par
Moreover, if $\mu \in \{(4),(3,1),(6),(5,1),(3,3),(3,2,1),(8),(5,3)\}$ and $(X,\omega)$
is not hyperelliptic, then no degenerate fiber of
the \teichm curve is hyperelliptic.  
\end{prop}
\par
The last conclusion does not hold for all strata. For instance, 
\teichm curves generated by a non-hyperelliptic flat surface 
in the stratum $\omoduli[3](2,1,1)$ always intersect
the hyperelliptic locus at the boundary, as we will see
later in the discussion for that stratum.
\par
As motivation for the proof, recall why a \teichm curve
generated by $(X,\omega)$ with $X$ hyperelliptic stays
within the corresponding locus of hyperelliptic flat surfaces.
The hyperelliptic
involution acts as $(-1)$ on all one-forms, hence on
$\omega$. In the flat coordinates of $X$ given by 
Re$(\omega)$ and $\Im(\omega)$, the hyperelliptic involution
acts by the matrix $-\Id$. The \teichm curve
is the $\SL(2,\RR)$-orbit of $(X,\omega)$ and $-\Id$
is in the center of $\SL(2,\RR)$. So if $(X,\omega)$
admits a hyperelliptic involution, so does
$A\cdot(X,\omega)$ for any $A\in \SL(2,\RR)$.
\par
\begin{proof}
Suppose the stable model $X_\infty$ of the degenerate fiber is irreducible 
of geometric genus $h$ with $(g-h)$ pairs of points $(p_i,q_i)$ identified. 
This stable curve $X_\infty$ being hyperelliptic means that there exists
a semi-stable curve birational to $X_\infty$ that admits a degree two admissible cover
of the projective line. In terms of admissible covers, this is yet equivalent to require that 
the normalization $X_n$ of $X_\infty$ is branched at $2h+2$ branch points over a main 
component (i.e.\ the image of the unique component not contracted under that passage to the stable model) with covering
group generated by an involution $\phi$ and, moreover, for each of the
$2(g-h)$ nodes there is a projective line intersecting $X_n$ in $p_i$ 
and $q_i=\phi(p_i)$  with two branch points.
\par
In the flat coordinates of $X_n$ given by $\omega$, the surface
consists of a compact surface $X_0$ with boundary of genus $h$ and 
$2(g-h)$ half-infinite cylinders (corresponding to the nodes) attached to 
the boundary of $X_0$. We may define $X_0$ canonically, by 
sweeping out the half-infinite cylinder at $p_i$ (or $q_i$) with lines
of slope equal to the residue (considered as element in $\RR^2$) of $\omega$
at $p_i$ until such a line hits a zero of $\omega$, i.e.\ a
singularity of the flat structure.
\par
With this normalization, the above discussion shows that
for irreducible stable curves the hyperelliptic involution exchanges
the half-infinite cylinders corresponding to $p_i$ and $q_i$ and
it defines an involution $\phi$ of $X_0$. As in the smooth case, 
$\phi$ acts as $-\Id$ on $X_0$.
\par
To obtain smooth fibers over the \teichm curve (in 
a neighborhood of $X_\infty$) one has to glue cylinders
of finite (large) height in place of the half-infinite
cylinders of appropriate ratios of moduli. The hypothesis
on $\phi$ acting on $X_0$ and on the half-infinite
cylinders implies that $\phi$ is a well-defined involution
on the smooth curves. Moreover, $\phi$ has two fixed
points in each of the finite cylinders and $2h+2$ fixed
points on $X_0$, making $2g+2$ fixed points in total.
This shows that the smooth fibers of the
\teichm curve are hyperelliptic.
\par
To complete the proof we have to consider the two-component
degenerations for $\mu \in \{(3,1),(5,1),(5,3)\}$ by Corollary~\ref{cor:degfibers}. In all these cases, 
the hyperelliptic involutions can neither exchange the
components (since the zeros are of different order) nor
fix the components (since the zeros are of odd order).
\par
For $\mu = (3,3)$ a hyperelliptic involution $\phi$ cannot fix the component, since
$3$ is odd. It cannot exchange the two components and exchange a
pair of half-infinite cylinders that belong to different nodes, 
since $\phi$ could then be used to define a non-trivial involution
for each component. This involution fixes the zeros and
this contradicts that $3$ is odd. If $\phi$ exchanges all
pairs of  half-infinite cylinders that belong to the same node, 
$\phi$ has two fixed points in each cylinder on the smooth
'opened up' surface. Now we can apply the same argument as
in the irreducible case to conclude that the 'opened up' 
flat surfaces are hyperelliptic as well.
\par
For $\mu = (3,2,1)$ a hyperelliptic involution can neither fix the
component with the (unique) zero of order three, since 3 is
odd, nor map it elsewhere, since the zeros are of different order.
\end{proof}
\par

\section{Lyapunov exponents, Siegel-Veech constants and slopes} \label{sec:backLyap}

\subsection{Lyapunov exponents}

Fix an $\SL_2(\RR)$-invariant, ergodic measure $m$ on $\omoduli$.
The Lyapunov exponents for the \teichm geodesic flow
on $\omoduli$ measure the logarithm of the growth rate of the 
Hodge norm of cohomology classes during parallel transport along 
the geodesic flow. More precisely, let $V$ be the restriction
of the real Hodge bundle (i.e.\ the bundle with fibers $H^1(X,\RR)$)
to the support $M$ of $m$. 
Let $S_t$ be the lift of the geodesic
flow to $V$ via the Gauss-Manin connection. Then Oseledec's theorem 
shows the existence of a filtration
$$V = V_{\lambda_1} \supset \cdots \supset V_{\lambda_k} \supset 0$$
by measurable vector subbundles with the property that, for almost all $p \in M$ 
and all $v \in V_p \setminus \{0\}$, one has
$$||S_t(v)|| = {\rm exp}(\lambda_i t + o(t)),$$
where $i$ is the maximum value such that $v$ is in the fiber of $V_i$
over $p$, i.e.\ $v\in (V_i)_p$. 
The numbers $\lambda_i$ for $i=1,\ldots,k\leq {\rm rank}(V)$
are called  the {\em  Lyapunov exponents of $S_t$}. Note
that these exponents are unchanged if we replace
the support of $m$ by a finite unramified covering with
a lift of the flow and the pullback of $V$. We adopt the convention
to repeat the exponents according to the
rank of $V_i/V_{i+1}$ such that we will always
have $2g$ of them, possibly some of them equal.
Since $V$ is symplectic, the spectrum is symmetric, i.e.\
$\lambda_{g+k} = - \lambda_{g-k+1}$. 
The reader may consult \cite{forni06} or \cite{zorich06} for a more 
detailed introduction to this subject. 
\par
Most of our results will be about the {\em sum} of Lyapunov exponents defined as 
$$L = \sum_{i=1}^g \lambda_i.$$
This sum depends, of course, on the measure $m$ chosen and we occasionally write $L(m)$ to emphasize
this dependence. In particular, one defines Lyapunov exponents for an $\SL_2(\RR)$-invariant suborbifold 
of $\omoduli$ carrying such a measure $m$. We will focus on the case of a \teichm curve $C$. Consequently, we 
use $L(C)$ to denote the {\em sum} of its Lyapunov exponents. 
\par
The bridge between the 'dynamical' definition of Lyapunov exponents
and the 'algebraic' method applied in the sequel is given by the
following result. Note that if the VHS splits into direct summands
one can apply Oseledec's theorem to the summands individually. The
full set of Lyapunov exponents is the union (with multiplicity) 
of the Lyapunov exponents of the summands.
\par
\begin{theorem} [\cite{kontsevich}, \cite{kontsevichzorich}, \cite{bouwmoel}] \label{thm:Lyapviadeg}
If the VHS over the \teichm curve contains a sub-VHS $\WW$ 
of rank $2k$, then the sum of the $k$ corresponding non-negative
Lyapunov exponents equals 
$$ \sum_{i=1}^k \lambda_i^{\WW} = \frac{2 \deg \WW^{(1,0)}}{2g(\bB)-2 + |\Delta|} ,$$
where $\WW^{(1,0)}$ is the $(1,0)$-part of the Hodge-filtration of the
vector bundle associated with $\WW$. In particular, we have 
$$ \sum_{i=1}^g \lambda_i =  \frac{2 \deg f_* \omega_{\bcX/\bB}}{2g(\bB)-2 + |\Delta|}. $$
\end{theorem}
\par

\subsection{Lyapunov exponents for loci of hyperelliptic flat surfaces}

We recall a result of \cite[Section 2.3]{ekz} that deals 
with the sum of Lyapunov exponents for \teichm curves 
generated by hyperelliptic curves, more
generally for any invariant measure on loci of hyperelliptic flat surfaces.
It implies immediately that hyperelliptic strata are non-varying.
\par
\begin{theorem}[\cite{ekz}] \label{thm:calchyp}
Suppose that $M$ is a regular $\SL_2(\RR)$-invariant suborbifold in a
locus of hyperelliptic flat surfaces of some stratum $\omoduli[g](m_1,\dots,m_k)$. 
Denote by $(d_1,\dots,d_s)$ the orders of singularities of the
underlying quadratic differentials on the quotient projective line.
\par
Then the sum of Lyapunov exponents for $M$ is 
$$
L(M)
\ = \
\cfrac{1}{4}\,\cdot\,\sum_{\substack{j \text{ such that}\\
d_j \text{ is odd}}}
\cfrac{1}{d_j+2}\,.
$$
where, as usual, we associate the order $d_i=-1$ to simple poles.
\end{theorem}
\par
\begin{cor} \label{cor:hypvalues}
Hyperelliptic strata are non-varying. For a \teichm
curve $C$ generated by $(X,\omega)$ we have
\begin{equation}
\begin{aligned}
L(C) = \frac{g^2}{2g-1} \quad  \text{and} \quad s(C) = 8 + \frac{4}{g} \quad & 
\text{if} \quad (X,\omega) \in \omoduli[g]^\hyp(2g-2), \\
L(C) = \frac{g+1}{2}  \quad  \text{and} \quad s(C) =  8 + \frac{4}{g} \quad & 
\text{if} \quad (X,\omega) \in \omoduli[g]^\hyp(g-1,g-1). 
\end{aligned}
\end{equation}
\end{cor}
\par

\subsection{Siegel-Veech constants, slopes and the sum of Lyapunov
exponents}\label{sec:svsl}

Write $\mu = (m_1, \ldots, m_k)$ for a partition of $2g-2$. 
Let $c_{\mu}$ denote the {\em (area) Siegel-Veech constant} of (the connected
component of) the stratum   
$\omoduli(\mu)$. Roughly speaking, $c_{\mu}$ measures the growth rate of the weighted sum of 
cylinders of length at most $T$ on a flat surface $(X,\omega)$ in $\omoduli(\mu)$. 
The weight for each horizontal cylinder is given by its height/length. 
Similarly, one can define the {\em Siegel-Veech constant} $c(C)$ 
for a \teichm curve $C$, or more generally for any $\SL_2(\RR)$-invariant suborbifold 
in $\omoduli(\mu)$ (see \cite{ekz} and \cite{emz} for a comprehensive introduction 
to Siegel-Veech constants). 
\par
Let $\kappa_{\mu}$ be a constant 
$$\kappa_{\mu} = \frac{1}{12}\sum_{i=1}^k \frac{m_i(m_i+2)}{m_i+1}$$ 
determined by the signature of the stratum. The Siegel-Veech constant and the sum of Lyapunov 
exponents are related as follows. The condition of regularity in the
next theorem is a technical notion that holds for all known examples of invariant suborbifolds and 
is expected to hold generally (see \cite[Section 1.5]{ekz} for more details).
\par
\begin{theorem}[\cite{ekz}] \label{thm:ekz}
For any regular $\SL_2(\RR)$-invariant finite
measure $m$ on the stratum $\omoduli[g](\mu)$ we have
\begin{equation}
\label{eq:Lkappac} 
L(m) = \kappa_\mu + c(m).
\end{equation}
In particular the regularity and hence the equality hold for the measure with support equal to a connected
component of $\omoduli[g](\mu)$ and for the measure supported on a \teichm curve.
\end{theorem}
\par
For any given \teichm curve at a time this theorem allows to calculate
the sum of Lyapunov exponents. It suffices to calculate the cusps (in practice,
e.g.\ for a square-tiled surface, this amounts to calculating the Veech group)
and to evaluate the (area) Siegel-Veech contribution of the cusp.
\par
Let $s(C)$ be the {\em slope} of a Teichm\"{u}ller curve 
$C$ defined by 
$$s(C) = \frac{\bC\cdot \delta_0}{\bC\cdot \lambda}.$$ 
Since a Teichm\"{u}ller curve generated by a flat surface does not intersect $\delta_i$ in $\barmoduli$ for $i > 0$ (see Corollary~\ref{cor:zerosatinf}),
its slope can also be defined as 
$$ s(C) = \frac{\bC\cdot \delta}{\bC\cdot \lambda}, $$
where $\delta = \sum_{i=0}^{[g/2]} \delta_i$ is the total boundary divisor. The latter is more commonly used for the slope of an arbitrary one-dimensional family of stable genus $g$ curves.  
\par
Given a \teichm curve $C$, one has to understand {\em only one} of
the quantities $L(C)$, $c(C)$ and $s(C)$, because of the relation
in Theorem~\ref{thm:ekz} and another relation as follows. 
\begin{prop}
For a Teichm\"{u}ller curve $C$ generated by a flat surface, we have 
\begin{equation}
\label{eq:Lcs}
s(C) = \frac{12c(C)}{L(C)} = 12 - \frac{12 \kappa_\mu}{L(C)}.  
\end{equation}
\end{prop}
\begin{proof}
This is a consequence of the Noether formula 
$$12 \lambda = \delta + f_{*}(c^2_1(\omega_{\bcX/\bC})),$$ 
as shown in \cite[Theorem 1.8]{chenrigid}.
\par
We can also directly see how it works. By Theorem~\ref{thm:Lyapviadeg} we know 
$\chi\cdot L(C) = 2 \deg \lambda$, where $\chi = 2g(C) - 2 + |\Delta|$. Using the Noether formula, the class 
of $\omega_{\bcX/\bC}$ in the proof of 
Proposition~\ref{prop:intomega} and Theorem~\ref{thm:ekz}, we can derive that $6 \chi\cdot c(C) = \deg \delta$. 
Hence the equality~\eqref{eq:Lcs} follows immediately.  
\end{proof}
\par
Now our strategy becomes clear. In order to show a stratum is non-varying, it suffices to show all \teichm curves in that stratum 
are disjoint from an effective divisor on $\barmoduli[g]$, hence they all have the same slope as that of the divisor. Then by~\eqref{eq:Lcs}, they have the same sum of Lyapunov exponents as well. We summarize this idea as follows. 
\begin{lemma}
\label{lem:slopenonvarying}
Let $D$ be an effective divisor on $\barmoduli$. Suppose the closures of all \teichm curves $\bC$ generated
by flat surfaces in a fixed stratum do not intersect $D$. 
Then they have the same slope  
$s(C)  = s(D)$. In particular, the sums of Lyapunov exponents are the same for these \teichm curves. 
\end{lemma}
\par
\begin{proof}
Recall the slope of an effective divisor defined in Section~\ref{sec:picmoduli}. Suppose $D$ has class
$a\lambda - \sum_{i=0}^{[g/2]} b_i \delta_i$. Then it has slope $s(D) = a / b_0$. Since $\bC\cdot D = 0$ and 
$\bC\cdot \delta_i = 0$
for $i > 0$ (Corollary~\ref{cor:zerosatinf}), we conclude that   
 $$ \frac{\bC\cdot \delta_0}{\bC\cdot \lambda} = \frac{a}{b_0}, $$
hence $s(C) = s(D)$. Since the slopes are non-varying, so are the sums of Lyapunov exponents 
and the Siegel-Veech constants for those \teichm curves, according to~\eqref{eq:Lcs}. 
\end{proof}
The same argument can help us find upper bounds for the slope as well as for the sum of Lyapunov exponents for a \teichm curve. 
\begin{lemma}
\label{lem:upperbound}
Let $D$ be an effective divisor on $\barmoduli$. Suppose a \teichm curve $C$ is not contained in $D$. Then we have   
$s(C) \leq s(D)$.
\end{lemma}
\begin{proof}
Suppose $D$ has class $a\lambda - \sum_{i=0}^{[g/2]} b_i \delta_i$. By assumption we have $C\cdot D\geq 0$. It implies that 
$a (\bC\cdot \lambda) - b_0 (\bC\cdot \delta_0) \geq 0$, hence 
$$s(C) = \frac{\bC\cdot \delta_0}{\bC\cdot \lambda} \leq \frac{a}{b_0} = s(D).$$ 
\end{proof}
\par
In some cases it is \emph{not} possible to find an effective divisor on $\barmoduli$ to perform the disjointness argument (see e.g. the explanation in Section~\ref{sec:str22odd}). Alternatively, we have to consider moduli spaces of curves with marked points or spin structures. Consequently we need to know the intersection of \teichm curves with the classes $\omega_{i,\rel}$ 
introduced in Section~\ref{sec:picmoduli}. 
\par
Let $C$ be a \teichm curve generated by 
$(X,\omega) \in \omoduli[g](m_1,\ldots, m_k)$.
Let $B \to C$ be a finite unramified cover such that
the $m_i$-fold zero defines a section $\sigma_i$ (not only a multi-section)
with image $S_i$ of the pullback family $f: \cX \to \bB$.  
\par
\begin{prop} \label{prop:intomega}
If $f: \bB \to \barmoduli[{g,1}]$ is the lift of
a \teichm curve by marking the zero of order $m_i$, then 
$$ S_i^2 = \frac{- \chi}{2(m_i+1)},$$
where $\chi = 2g(\bB) -2 + |\Delta|$ and $\Delta$ is the set of cusps in $\bB$. 
In particular the 
intersection number with $\omega_{i,\rel}$, which is by definition equal to $-S_i^2$,
is given by $$ \bB\cdot \omega_{i,\rel} = \frac{\bB\cdot \lambda - 
(\bB\cdot \delta)/12}{ (m_i+1) \kappa_\mu}, $$
where $\kappa_{\mu} = \frac{1}{12}\sum_{j=1}^k\frac{m_j(m_j+2)}{m_j+1}.$
\end{prop}
\par
\begin{proof}
Let $\cL \subset f_* \omega_{\bcX/\bB}$ be the ('maximal Higgs', see \cite{moeller06})
line bundle whose fiber over the point corresponding to $[X]$ is $\CC \cdot \omega$, 
the generating differential of the \teichm curve. The property
'maximal Higgs' says by definition that 
\begin{equation} \label{eq:maxHiggs}
 \deg(\cL) = \chi/2.
\end{equation}
Let $S$ be the union of the sections $S_1,\ldots, S_k$. 
Pulling back the above inclusion to $\cX$ gives an exact sequence
$$ 0 \rightarrow f^* \cL \rightarrow \omega_{\bcX/\bB} \rightarrow \cO_S\left( \sum_{j=1}^k m_j S_j\right) \rightarrow 0,$$
since the multiplicities of the vanishing locus of the generating differential
of the \teichm curve are constant along the whole compactified \teichm curve.
This implies that $\omega_{\cX /\bB}$ is numerically equal to 
$$ f^* \cL + \sum_{j=1}^k m_j S_j.$$
By the adjunction formula we get
$$S_i^2 = -\omega_{\bcX /\bB} \cdot S_i = -m_i S_i^2 - \deg(\cL)$$
since the intersection product of two fibers of $f$ is zero. Together
with \eqref{eq:maxHiggs} we thus obtain the desired self-intersection formula.
\par
By Theorem~\ref{thm:Lyapviadeg} and the relation~\eqref{eq:Lcs}, we have 
$$\bB\cdot \lambda = \frac{\chi}{2}\cdot L, $$
$$\bB\cdot \delta = \frac{\chi}{2}\cdot (12L - 12\kappa_{\mu}). $$
Hence the second claimed formula follows, for $- S_i^2 = \bB\cdot \omega_{i,\rel}$ by definition. 
\end{proof}
\par
\begin{rem}
For square-tiled surfaces the self-intersection number of a section on the 
elliptic surface is not hard to calculate (\cite{Kd63}, recalled in \cite{moellerST}, 
and also in \cite[Theorem 1.15]{chencovers}). 
Pullback introduces the coefficient $m_i+1$ in the denominator. This shows the 
formula in the square-tiled case. The general case of the 
formula can also be shown by adapting the argument given in
 \cite[Theorem 12.2]{bainbridge07}, since there are $m_i+1$ ways to
split a singularity of order $m_i$.
\end{rem}
\par
If $L$ is non-varying for all \teichm curves (or just those generated by square-tiled surfaces) in a 
stratum $\omoduli[g](\mu)$, it implies that the sum of Lyapunov exponents for the whole stratum is equal to $L$. 
\begin{prop}\label{localglobal}
As the area (i.e. the degree of the torus coverings) approaches infinity, the limit of sums of Lyapunov exponents for \teichm curves generated by square-tiled surfaces in 
a stratum is equal to the sum of Lyapunov exponents for that stratum. 
\end{prop}
\par
This is due to the fact that square-tiled surfaces in a stratum are parameterized by 'lattice points' under the period coordinates (see e.g. \cite[Lemma 3.1]{eo}). 
Hence, their asymptotic behavior reveals information for the whole stratum (see e.g. \cite[Appendix A]{chenrigid} for a proof).

\section{Genus three} \label{sec:g=3}
In genus $3$ all the strata have non-varying sums of Lyapunov exponents except the principal stratum. We 
summarize the results in Table~\ref{cap:tabg3}. We also give a sharp upper bound for the sum of Lyapunov exponents for the principal stratum. 
\par
Let us first explain how to read the table. For example, the stratum $(2,2)^{\odd}$ is non-varying. The sum of Lyapunov exponents is 
equal to $\frac{5}{3}$ ($\approx 1.66666$) for both the stratum and every \teichm curve in the stratum. On the other hand, 
the principal stratum $(1,1,1,1)$ is varying. The sum of Lyapunov exponents for the whole stratum is $\frac{53}{28}$ ($\approx 1.89285$).  
The sharp upper bound for the sums of Lyapunov exponents of \teichm curves in this stratum is $2$. It can be attained, e.g. by \teichm curves 
in the locus of hyperelliptic flat surfaces as the image of $\cQ(2,2,-1^8)$ in the context of Theorem~\ref{thm:calchyp}. Later on we will use similar tables to encode the behavior of \teichm curves in other genera. 

\begin{figure}
    \centering      
$$
\begin{array}{|c|c||c|c|c|c|c|}

\hline
&&\multicolumn{5}{|c|}{}\\

\multicolumn{1}{|c|}{\text{Degrees}}&
\multicolumn{1}{|c||}{\text{Hyperelliptic}}&
\multicolumn{5}{|c|}{\text{Lyapunov exponents}}\\

\multicolumn{1}{|c|}{\text{of }}&
\multicolumn{1}{|c||}{\text{or spin}}&
\multicolumn{5}{|c|}{\text{}}\\

\cline{3-7}

\multicolumn{1}{|c|}{\text{zeros}}&
\multicolumn{1}{|c||}{\text{structure}}&
\multicolumn{2}{|c|}{\text{}}&
\multicolumn{3}{|c|}{\text{}}\\

\multicolumn{1}{|c|}{(d_1,\dots,d_n)}&
\multicolumn{1}{|c||}{}&
\multicolumn{2}{|c|}{\text{Component}}&
\multicolumn{3}{|c|}{\text{\teichm curves}}\\

\multicolumn{1}{|c|}{\text{}}&
\multicolumn{1}{|c||}{\text{}}&
\multicolumn{2}{|c|}{\text{}}&
\multicolumn{3}{|c|}{\text{}}\\

\cline{3-7}

\multicolumn{1}{|c|}{\text{}}&
\multicolumn{1}{|c||}{\text{}}&
\multicolumn{1}{|c|}{\approx}&
\multicolumn{1}{|c|}{\overset{g}{\underset{j=1}{\sum}} \lambda_j} &
\multicolumn{1}{|c|}{\approx}&
\multicolumn{1}{|c|}{\overset{g}{\underset{j=1}{\sum}} \lambda_j}&
\multicolumn{1}{|c|}{\text{Reference}}\\
[-\halfbls] &&&&&&\\
\hline &&&& \multicolumn{2}{|c|}{} & \\  [-\halfbls]
(4) & \text{hyperelliptic}  & 1.80000 & \frac{9}{5} 
& \multicolumn{2}{|c|}{\text{Non-varying}} & \text{Thm.~\ref{thm:calchyp}} \\ 
[-\halfbls] &&&&\multicolumn{2}{|c|}{} &\\
\hline &&&& \multicolumn{2}{|c|}{} & \\  [-\halfbls]
(4) & \text{odd} &  1.60000 & \frac{8}{5} 
& \multicolumn{2}{|c|}{\text{Non-varying}} & \text{Sec.~\ref{sec:4odd}}\\ 
[-\halfbls] &&&&\multicolumn{2}{|c|}{}&\\
\hline &&&& \multicolumn{2}{|c|}{} & \\  [-\halfbls]
(3,1) & -  & 1.75000 & \frac{7}{4} 
& \multicolumn{2}{|c|}{\text{Non-varying}} & \text{Sec.~\ref{sec:31}}\\ 
[-\halfbls] &&&&\multicolumn{2}{|c|}{}&\\
\hline &&&& \multicolumn{2}{|c|}{} & \\  [-\halfbls]
(2, 2) & \text{hyperelliptic} & 2.00000 & 2 
& \multicolumn{2}{|c|}{\text{Non-varying}} & \text{Thm.~\ref{thm:calchyp}}  \\ 
[-\halfbls] &&&&\multicolumn{2}{|c|}{}&\\
\hline &&&& \multicolumn{2}{|c|}{} & \\  [-\halfbls]
(2,2) & \text{odd} & 1.66666 & \frac{5}{3} 
& \multicolumn{2}{|c|}{\text{Non-varying}} & \text{Sec.~\ref{sec:str22odd}}\\ 
[-\halfbls] &&&&\multicolumn{2}{|c|}{}&\\
\hline &&&& \multicolumn{2}{|c|}{} & \\  [-\halfbls]
(2,1,1) & - & 1.83333 & \frac{11}{6}  
& \multicolumn{2}{|c|}{\text{Non-varying}} & \text{Sec.~\ref{sec:211}}\\ 
[-\halfbls] &&&&\multicolumn{2}{|c|}{}&\\
\hline &&&&&& \\ [-\halfbls]
(1,1,1,1) & - & 1.89285 & \frac{53}{28} & 2 & 2 & \cQ(2,2,-1^8)\\ 
[-\halfbls] &&&&&&\\ \hline
\end{array}
$$
\caption{Varying and non-varying sums in genus three} \label{cap:tabg3}
\end{figure}

\subsection{The stratum $\omoduli[3](4)^\odd$} \label{sec:4odd}

In the case  $\omoduli[3](4)^\odd$ the algorithm of \cite{emz} to
calculate Siegel-Veech constants for components of strata gives
$$ L_{(4)^\odd} = 8/5, \quad  s_{(4)^\odd} = 9, \quad c_{(4)^\odd} = 6/5. $$
\par
\begin{proof}[{Proof of Theorem~\ref{thm:g3main}, Case ${\omoduli[3]}(4)^\odd$}]
The connected components  ${\omoduli[3]}(4)^\odd$  and ${\omoduli[3]}(4)^\hyp$ are not 
only disjoint in $\omoduli[3]$, by Proposition~\ref{prop:Teichhypdisj} they are also disjoint
in $\obarmoduli[3]$. Hence a \teichm curve $\bC$ generated by a flat surface in this stratum do not 
intersect the hyperelliptic locus $H$ in 
$\barmoduli[3]$. Recall the divisor class of $H$ in \eqref{eq:classofH}.
By Lemma~\ref{lem:slopenonvarying} and $s(H) = 9$, we obtain that $s(C)=9$, hence $c(C) = 6/5$
and $L(C) = 8/5$ for all \teichm curves in this stratum 
using \eqref{eq:Lkappac} and \eqref{eq:Lcs}.
\end{proof}

\subsection{The stratum $\omoduli[3](3,1)$} \label{sec:31}


In the case  $\omoduli[3](3,1)$ we have
$$ L_{(3,1)} = 7/4, \quad  s_{(3,1)} = 9, \quad c_{(3,1)} = 21/16. $$
\par
\begin{proof}[{Proof of Theorem~\ref{thm:g3main}, Case ${\omoduli[3]}(3,1)$}]
As in the case of the stratum $\omoduli[3](4)^\odd$, a \teichm curve $C$ generated
by a flat surface in the 
stratum $(3,1)$ does not intersect the hyperelliptic
locus, not even at the boundary by Proposition~\ref{prop:Teichhypdisj}. Consequently 
we can apply the same disjointness argument as in the preceding case. 
Since $s(H)=9$, we obtain that $s(C) =9$, hence $c(C) = 21/16$
and $L(C) = 7/4$ using \eqref{eq:Lkappac} and \eqref{eq:Lcs}.

\end{proof}

\subsection{The stratum $\omoduli[3](2,2)^\odd$} \label{sec:str22odd}

In the case  $\omoduli[3](2,2)^\odd$, we have
$$ L_{(2,2)^\odd} = 5/3, \quad  s_{(2,2)^\odd} = 44/5, \quad c_{(2,2)^\odd} = 11/9. $$
\par
Note that on $\barmoduli[3]$ the smallest slope of an effective divisor is $9$ 
(attained by the divisor $H$ of hyperelliptic curves, see \cite[Theorem 0.4]{hamoslope}). In order to show a \teichm curve $\bC$ 
has $L(C) = 44/5$ in this stratum, we \emph{cannot} use an effective divisor $D$ on $\barmoduli[3]$ such that 
$\bC\cdot D = 0$. Otherwise by Lemma~\ref{lem:slopenonvarying} 
it would imply that $s(C) = s(D) \geq 9 > 44/5$. Instead, we have to use another moduli space parameterizing 
curves with some additional structure. Here the stratum is distinguished by spin structures, hence it is natural to consider the 
spin moduli space introduced in Section~\ref{sec:spinmod}. 

Since the same idea will also be applied to the stratum $\omoduli[4](2,2,2)^\odd$, we first consider the general 
case $\omoduli[g](2,\ldots, 2)^\odd$. 
Let $C$ be a \teichm curve  generated by a flat surface $(X, \omega)$ in $\obarmoduli(2,\ldots, 2)^\odd$ such that 
 $\div(\omega) = 2 \sum_{i=1}^{g-1} p_i$ for distinct points $p_i$. Using $\eta = \sum_{i=1}^{g-1} p_i$ 
 as an odd theta characteristic, we can map $\bC$ to $\Sgo$.
\par
Let $\Zg$ be the divisor on $\Sgo$ parameterizing $(X,\eta)$ such that the odd theta characteristic 
$\eta$ satisfies 
$$\eta \sim \OO_X(2p_1+p_2+\cdots+p_{g-2}).$$ 
The class of $\Zg$ was calculated in \cite[Theorem 0.4]{FV}: 
$$ \Zg =  (g+8) \lambda - \frac{g+2}{4}\alpha_0 - 2\beta_0 - \sum\limits_{i=1}^{[g/2]}2(g-i)\alpha_i - \sum_{i=1}^{[g/2]} 2i\beta_i,$$
where $\lambda$ is the pullback of the $\lambda$-class on $\Mg$ and 
$\alpha_i, \beta_i$ are two different boundary divisors of $\Sgo$ over $\delta_i$ for each $0\leq i\leq [g/2]$. Since 
\teichm curves do not intersect $\delta_i$ for $i > 0$, we focus on $\delta_0$ and its inverse images $\alpha_0, \beta_0$ only. 
By definition \cite[Section 1.2]{FV}, a spin curve $Y$ in $\beta_0$ possesses an exceptional component $E$, i.e. a rational curve that meets the rest of $Y$ at two nodes, and the theta characteristic $\eta$ has degree $1$ restricted to $E$. In particular, $Y$ cannot be parameterized in the boundary of a \teichm curve, since non of the zeros $p_i$ will lie on $E$ due to the fact that $\omega_Y|_E$ has degree $0$ and Corollary~\ref{cor:zerosatinf}. 
We thus conclude the following. 
\begin{lemma}
\label{lem:spindisjoint}
In the above setting, $\bC$ does not intersect any boundary components of $\Sgo$ except $\alpha_0$. 
\end{lemma}
\par 
We are interested in the case when an odd theta characteristic does not have extra sections. 
\begin{prop}
\label{prop:spinslope}
In the above setting, suppose every odd theta characteristic $\eta$ parameterized in $\mathcal{S}_g^{-}$ union $\alpha_0$ 
satisfies $h^0(\eta) = 1$. Then the slope of $C$ is 
$$s(C) = \frac{4(g+8)}{g+2}. $$
\end{prop}
\begin{proof}
In $\mathcal{S}_g^{-}$ union $\alpha_0$, $\Zg$ can be identified as the stratum 
$\obarmoduli(4, 2,\ldots, 2)^\odd$, since the unique section of $\eta$ determines 
the zeros of the corresponding Abelian differential. Consequently by 
Proposition~\ref{prop:Teichbddisj}, the image of $\bC$ in $\Sgo$ does not intersect $\Zg$.
Then we have 
$$ 0 = \bC\cdot \Zg = \bC\cdot \Big((g+8) \lambda - \frac{g+2}{4}\alpha_0\Big), $$
since $\bC$ does not intersect any boundary components except $\alpha_0$ by Lemma~\ref{lem:spindisjoint}. Note that 
$\pi^{*}\lambda = \lambda$ and $\pi^{*}\delta_0 = \alpha_0+2\beta_0$, where $\pi: \Sgo\to \barmoduli$ is 
the morphism forgetting the spin structure and stabilizing the curve (see \cite[Section 1.2]{FV}).
By the projection formula we have
$$ 0 = (\pi_{*}\bC)\cdot \Big((g+8)\lambda - \frac{g+2}{4}\delta_0\Big). $$
The desired formula follows right away. 
\end{proof}
We remark that the assumption $h^0(\eta) = 1$ for all $\eta$ is rather strong and seems to hold only in low genus, as a consequence 
of Clifford's theorem (Theorem~\ref{thm:clifford}). 

\begin{proof}[{Proof of Theorem~\ref{thm:g3main}, Case ${\omoduli[3]}(2,2)^\odd$}]
For $g=3$ an odd theta characteristic cannot have three or more
sections by Theorem~\ref{thm:clifford}. Hence Proposition~\ref{prop:spinslope}
applies and we obtain that $s(C) = 44/5$. 
\end{proof}
\par
Note that this argument does not distinguish between hyperelliptic
and non-hyperelliptic curves and for double covers of $\cQ(1,1,-1^6)$ 
the result is in accordance with Theorem~\ref{thm:calchyp}.

\subsection{The stratum $\omoduli[3](2,1,1)$} \label{sec:211}

In the case  $\omoduli[3](2,1,1)$ we have
$$ L_{(2,1,1)} = 11/6, \quad  s_{(2,1,1)} = 98/11, \quad c_{(2,1,1)} = 49/36. $$
\par
Since $s_{(2,1,1)} = 98/11 < 9$, by the same reason as in the preceding section, 
one cannot verify this non-varying slope by disjointness with an effective divisor on $\barmoduli[3]$. This time the flat surfaces
have three zeros, hence it is natural to consider the moduli space of pointed curves by marking some of the zeros. Consequently we will seek 
certain pointed Brill-Noether divisors to perform the disjointness argument. 
\par
First we lift a \teichm curve generated by a flat surface
in this stratum to $\barmoduli[{3,2}]$ by marking the double zero $p$ as
the first point and one of the simple zeros $q$ as the second point. In fact, 
we can do so after passing to a double covering of $C$ where the simple zeros
$q$ and $r$ can be distinguished. This double covering is unramified, 
since by definition of a \teichm curve the zeros never collide.
Moreover, the slope and hence the sum of Lyapunov exponents are unchanged
by passing to an unramified double covering. For simplicity we will continue to call $C$ the \teichm curve we work with.
\par 
\begin{prop}
In the above setup, the \teichm curve $\bC$ does not intersect the pointed
Brill-Noether divisor $BN^1_{3, (1,2)}$ on $\barmoduli[{3,2}]$. 
\end{prop}
\par
\begin{proof}
Recall that $BN^1_{3, (1,2)}$ parameterizes pointed curves $(X, p, q)$ that possess a $g^1_3$ containing 
$p+2q$ as a section. Suppose that $(X,\omega)$ is in the intersection
of $\bC$ and $BN^1_{3, (1,2)}$. Since $h^0(\OO_X(p+2q)) = 2$ and $p, q, r$ are distinct, we
obtain that $h^0(\OO_X(p+r-q)) = 1$ by Riemann-Roch and then $h^0(\OO_X(p+r)) = 2$.
If $X$ is smooth, then $X$ is hyperelliptic and $p, r$ are conjugate. 
But $\omega_X \sim \OO_{X}(2p+q+r)$, so $p,q$ are also conjugate, contradiction.
For singular $X$ we deduce from $h^0(\OO_X(p+r)) = 2$ that $p$ and $r$ are in
the same component $X_0$ of $X$. This component admits an involution $\phi$
that acts on the set of zeros of $\omega_X|_{X_0}$. But $p$ and $r$ have different orders, so they 
cannot be conjugate under $\phi$, leading to a contradiction. 
\end{proof}
\par
\begin{proof}[{Proof of Theorem~\ref{thm:g3main}, Case ${\omoduli[3]}(2,1,1)$}]
By the proposition and the divisor class of $BN^1_{3, (1,2)}$ in \eqref{eq:BN(3,1,2)},
we obtain that 
$$ \bC\cdot (-\lambda + \omega_{1,\rel} + 3\omega_{2,\rel} ) = 0. $$
Since $\kappa_{(2,1,1)} = 17/36$, using Proposition~\ref{prop:intomega}, we have 
$$ \bC\cdot \omega_{1,\rel} =  \frac{\bC\cdot\lambda - (\bC\cdot\delta)/12}{17/12}, $$
$$ \bC\cdot \omega_{2,\rel} =  \frac{\bC\cdot\lambda - (\bC\cdot\delta)/12}{17/18}. $$ 
Plugging in the above, we obtain that $s(C) = 98/11$ and 
the values of $L(C), c(C)$ follow from \eqref{eq:Lkappac} and \eqref{eq:Lcs}.
\end{proof}
\par
\begin{rem}
Alternatively, the theorem can be deduced by showing that for a \teichm
curve $\bC$ its intersection loci with the hyperelliptic locus $H$ in $\barmoduli[3]$
and with the Weierstrass divisor $W$ in $\barmoduli[3, 1]$ are the same.
We show that these intersections are set-theoretically equal. A complete
proof via this method would need to verify that the intersection multiplicities
with the two divisors coincide.
\par
Hyperelliptic flat surfaces
in this stratum are obtained as coverings from the stratum $\cQ(2,1,-1^7)$ and we deduce from 
Theorem~\ref{thm:calchyp} that $L^\hyp_{(2,1,1)} = 11/6$. 
Hence we may assume that $\bC$ is not entirely in the hyperelliptic locus. 
If $\bC$ intersects $W$ at a point $(X, 2p, q, r)$, where $p$ is the marked point, 
then $p$ is a Weierstrass point. Hence we have $ 2p + q + r \sim  3p + s$ for some $s$ in $X$. Consequently $q+r \sim p + s$  and $X$ 
must be hyperelliptic. On the other hand, suppose $\bC$ intersects $H$ 
at a point $(X, 2p, q, r)$. Since $X$ is hyperelliptic, $p$ must be a 
Weierstrass point. 
\end{rem}
\par
\begin{cor}
A \teichm curve generated by a non-hyperelliptic flat surface 
in $\omoduli[3](2,1,1)$ does intersect the hyperelliptic locus $H$ at the boundary.
\end{cor}
\par
\begin{proof}
If the statement was false for some \teichm curve $C$, we
would have $\bC\cdot H = 0$, hence $s(C) = s(H) = 9$, contradicting $s(C) = 98/11$.
\end{proof}

\subsection{Varying sum in the stratum $\omoduli[3](1,1,1,1)$}

We show by example that the sum of Lyapunov exponents in the principal stratum
in $g=3$ is varying, even modulo the hyperelliptic locus. In the case
$\omoduli[3](1,1,1,1)$, the algorithm of \cite{emz} to
calculate Siegel-Veech constants for components of strata gives 
$$ L_{(1,1,1,1)} = 53/28, \quad  s_{(1,1,1,1)} = 468/53, \quad c_{(1,1,1,1)} = 39/28. $$
\par
\begin{exs}
The 'eierlegende Wollmilchsau', the square-tiled surface given by the 
permutations $(\pi_r=(1234)(5678), \pi_u = (1836)(2745))$ (see \cite{forni06} and \cite{moellerST}), 
generates  a \teichm curve $C$ with $L(C) = 1$.
\par
The square-tiled surface given by the 
permutations 
$$(\pi_r=(1234)(5)(6789), \pi_u = (1)(2563)(4897))$$  
generates a \teichm curve $C$ with $L(C) = 2$.
It attains the upper bound given by Theorem~\ref {thm:g3main}, but it 
is not hyperelliptic.
\par
There exist square-tiled surfaces in this stratum whose associated \teichm curves $C$ have
$$ L(C) \in \{1, 3/2, 5/3, 7/4, 9/5, 11/6, 19/11, 33/19, 83/46, 544/297\}.$$
\end{exs}
\par
\begin{proof}[{Proof of Theorem~\ref{thm:g3main}, Case ${\omoduli[3]}(1,1,1,1)$}]
\teichm curves in the locus of hyperelliptic flat surfaces (i.e. the image of 
$\cQ(2,2,-1^8)$) in $\omoduli[3](1,1,1,1)$ have $L = 2$ by Theorem~\ref{thm:calchyp}. 
\par
If the \teichm curve $C$ is not contained in the hyperelliptic locus,
then $\bC \cdot H \geq 0$, or equivalently $s(C) \leq s(H) = 9$ by Lemma~\ref{lem:upperbound}. 
Using $\kappa_{(1,1,1,1)} = 1/2$ this implies $L(C) \leq 2$.
\par
For the last statement in the theorem recall that an double
cover of a genus two curve is always hyperelliptic (e.g.\ \cite{FaHyp}).
\end{proof}

\section{Genus four}\label{sec:g=4}

In genus $4$ we summarize the non-varying sums of Lyapunov exponents and upper bounds for varying sums 
in Table~\ref{cap:tabg4}. 

\begin{figure}
    \centering      
$$
\begin{array}{|c|c||c|c|c|c|c|}

\hline
&&\multicolumn{5}{|c|}{}\\

\multicolumn{1}{|c|}{\text{Degrees}}&
\multicolumn{1}{|c||}{\text{Hyperelliptic}}&
\multicolumn{5}{|c|}{\text{Lyapunov exponents}}\\

\multicolumn{1}{|c|}{\text{of }}&
\multicolumn{1}{|c||}{\text{or spin}}&
\multicolumn{5}{|c|}{\text{}}\\

\cline{3-7}

\multicolumn{1}{|c|}{\text{zeros}}&
\multicolumn{1}{|c||}{\text{structure}}&
\multicolumn{2}{|c|}{\text{}}&
\multicolumn{3}{|c|}{\text{}}\\

\multicolumn{1}{|c|}{(d_1,\dots,d_n)}&
\multicolumn{1}{|c||}{}&
\multicolumn{2}{|c|}{\text{Component}}&
\multicolumn{3}{|c|}{\text{\teichm curves}}\\

\multicolumn{1}{|c|}{\text{}}&
\multicolumn{1}{|c||}{\text{}}&
\multicolumn{2}{|c|}{\text{}}&
\multicolumn{3}{|c|}{\text{}}\\

\cline{3-7}

\multicolumn{1}{|c|}{\text{}}&
\multicolumn{1}{|c||}{\text{}}&
\multicolumn{1}{|c|}{\approx}&
\multicolumn{1}{|c|}{\overset{g}{\underset{j=1}{\sum}} \lambda_j} &
\multicolumn{1}{|c|}{\approx}&
\multicolumn{1}{|c|}{\overset{g}{\underset{j=1}{\sum}} \lambda_j}&
\multicolumn{1}{|c|}{\text{Reference}}\\
[-\halfbls] &&&&&& \\
\hline &&&&\multicolumn{2}{|c|}{}& \\ [-\halfbls]
(6) & \text{hyperelliptic}  & 2.28571 & \frac{16}{7} &
\multicolumn{2}{|c|}{\text{Non-varying}} & \text{Thm.~\ref{thm:calchyp}}  \\
[-\halfbls] &&&&\multicolumn{2}{|c|}{}&\\
\hline &&&&\multicolumn{2}{|c|}{}& \\ [-\halfbls]
(6) & \text{even}  & 2.00000 & 2
& \multicolumn{2}{|c|}{\text{Non-varying}} & \text{Sec.~\ref{sec:6even}}\\ 
[-\halfbls] &&&&\multicolumn{2}{|c|}{}&\\
\hline &&&&\multicolumn{2}{|c|}{}& \\ [-\halfbls]
(6) & \text{odd} &  1.85714  & \frac{13}{7} 
& \multicolumn{2}{|c|}{\text{Non-varying}} & \text{Sec.~\ref{sec:6odd}}\\ 
[-\halfbls] &&&&\multicolumn{2}{|c|}{}&\\
\hline &&&&\multicolumn{2}{|c|}{}& \\ [-\halfbls]
(5, 1) & - & 2.00000 & 2
& \multicolumn{2}{|c|}{\text{Non-varying}} & \text{Sec.~\ref{sec:51}}\\ 
[-\halfbls] &&&&\multicolumn{2}{|c|}{}&\\
\hline &&&&\multicolumn{2}{|c|}{}& \\ [-\halfbls]
(4,2) & \text{even} & 2.13333 & \frac{32}{15}  
& \multicolumn{2}{|c|}{\text{Non-varying\ ?}} & - \\ 
[-\halfbls] &&&&\multicolumn{2}{|c|}{}&\\
\hline &&&&\multicolumn{2}{|c|}{}& \\ [-\halfbls]
(4,2) & \text{odd} & 1.93333 & \frac{29}{15}  
& \multicolumn{2}{|c|}{\text{Non-varying\ ?}} & - \\ 
[-\halfbls] &&&&\multicolumn{2}{|c|}{}&\\
\hline &&&&\multicolumn{2}{|c|}{}& \\ [-\halfbls]
(3,3) & \text{hyperelliptic} & 2.50000 & \frac{5}{2} &
\multicolumn{2}{|c|}{\text{Non-varying}} & \text{Thm.~\ref{thm:calchyp}}  \\
[-\halfbls] &&&&\multicolumn{2}{|c|}{}&\\
\hline &&&&\multicolumn{2}{|c|}{}& \\ [-\halfbls]
(3,3) & \text{non-hyp} &  2.00000 & 2
& \multicolumn{2}{|c|}{\text{Non-varying}} & \text{Sec.~\ref{sec:33nh}}\\ 
[-\halfbls] &&&&\multicolumn{2}{|c|}{}&\\
\hline &&&&\multicolumn{2}{|c|}{}& \\ [-\halfbls]
(3,2,1) & - &  2.08333 & \frac{25}{12}
& \multicolumn{2}{|c|}{\text{Non-varying}} & \text{Sec.~\ref{sec:321}}\\ 
[-\halfbls] &&&&\multicolumn{2}{|c|}{}&\\
\hline &&&&\multicolumn{2}{|c|}{}& \\ [-\halfbls]
(2, 2, 2) & \text{odd} & 2.00000 & 2 
& \multicolumn{2}{|c|}{\text{Non-varying}} & \text{Sec.~\ref{sec:222odd}}  \\
[-\halfbls] &&&&\multicolumn{2}{|c|}{}&\\
\hline &&&&&& \\ [-\halfbls]
(2, 2, 2) & \text{even} & 2.28571 & \frac{166}{75} & 2.333333 & \frac{7}{3} & \cQ(3,1,-1^8)\\
[-\halfbls] &&&&&&\\
\hline &&&&&& \\ [-\halfbls]
(4, 1, 1) & - & 2.06727 & \frac{1137}{550} & 1.96792 & \frac{1043}{530} & \text{Eq.\ \eqref{eq:411}} \\
[-\halfbls] &&&&&&\\
\hline &&&&&& \\ [-\halfbls]
(2^2, 1^2) & - &  2.13952 & \frac{5045}{2358} & 1.91666 & \frac{23}{12} & \cQ(2,1,1, -1^7)\\
[-\halfbls] &&&&&&\\
\hline &&&&&& \\ [-\halfbls]
(3,1^3) & - & 2.12903 & \frac{66}{31} & 2.11523 & \frac{514}{243}  &  \text{Eq.\ \eqref{eq:3111}}\\
[-\halfbls] &&&&&&\\
\hline &&&&&& \\ [-\halfbls]
(2, 1^4) & - &  2.18333 & \frac{131}{60} & 2.80000 & \frac{14}{5} & \cQ(3,2,2,-1^{11})\\
[-\halfbls] &&&&&&\\
\hline &&&&&& \\ [-\halfbls]
(1^6) & - &  2.22546 & \frac{839}{377} & 2.50000 & \frac{5}{2} & \cQ(2,2,2,-1^{10})\\
[-\halfbls] &&&&&&\\ \hline
\end{array}
$$
\caption{Varying and non-varying sums in genus four} \label{cap:tabg4}
\end{figure}


\subsection{The stratum $\omoduli[4](6)^\even$} \label{sec:6even}

In the case  $\omoduli[4](6)^\even$, we have
$$ L_{(6)^\even} = 14/7 , \quad  s_{(6)^\even} = 60/7, \quad c_{(6)^\even} = 10/7. $$
\par
\begin{prop} \label{prop:intthetanull}
Let $C$ be a \teichm curve generated by $(X,\omega) \in \omoduli[4](6)^\even$,
lifted to $\barmoduli[{4,1}]$ using the zero of $\omega$. Then $\bC$
does not intersect the theta-null divisor $\Theta$ in $\barmoduli[{4,1}]$. 
\end{prop}
\par
\begin{proof}
Recall that the divisor $\Theta\subset \barmoduli[{4,1}]$ parameterizes curves that admit an odd theta 
characteristic whose support contains the marked point. Suppose the stable pointed curve $(X,p)$ lies in 
the intersection of $\bC$ and $\Theta$. Then there exists an odd theta characteristic $\eta$ on $X$ with a section 
$t \in H^0(\eta)$ such that 
$\div(t)= p + q + r$ for some $q,r$ not both equal to $p$. Denote by $\cL = \OO_X(3p)$ the line bundle corresponding 
to the even theta characteristic with a section $s \in H^0(\cL)$ given by $3p$. Since $\eta^{\otimes 2}\sim \omega_X\sim \cL^{\otimes 2}$, 
the function $s^2t^{-2}$ implies that $4p\sim 2q + 2r$ on $X$,  hence $p$ is not a base point of $|\cL(p)| = |\OO_X(4p)|$. 
Consequently we have $h^0(\cL(p)) = 1 + h^0(\cL) \geq 3$. By $\omega_X\sim \OO_X(6p)$ and Riemann-Roch, 
$h^0(\OO_X(2p)) = h^0(\OO_X(4p)) -1 \geq 2$. Since $X$ is irreducible by
Corollary~\ref{cor:zerosatinf}, this implies that $X$ is hyperelliptic and $p$ is a Weierstrass point. 
It contradicts the disjointness of the hyperelliptic locus and this component
in $\barmoduli[4]$ by Proposition~\ref{prop:Teichhypdisj}.
\end{proof}
\par
\begin{proof}[{Proof of Theorem~\ref{thm:g4main}, Case ${\omoduli[4]}(6)^\even$}]
Using the proposition and the class of the theta-null divisor $\Theta$ in \eqref{eq:thetanull}, we obtain that
$$ \bC \cdot (30\lambda + 60\omega_\rel - 4\delta_0) = 0.$$
Using Proposition~\ref{prop:intomega} we know
$$ \bC\cdot \omega_\rel = \frac{\bC\cdot \lambda - (\bC\cdot \delta)/12}{4}.$$
It now suffices to plug this in and use \eqref{eq:Lkappac} and \eqref{eq:Lcs}.
\end{proof}


\subsection{The stratum $\omoduli[4](6)^\odd$} \label{sec:6odd}

In the case  $\omoduli[4](6)^\odd$, we have
$$ L_{(6)^\odd} = 13/7, \quad  s_{(6)^\odd} = 108/13, \quad c_{(6)^\odd} = 9/7. $$
\par
\begin{prop}
Let $C$ be a \teichm curve generated by $(X,\omega)\in \omoduli[4](6)^\odd$ lifted to $\barmoduli[{4,1}]$ using the zero of $\omega$. Then 
$C$ does not intersect the pointed Brill-Noether divisor $BN^1_{3,(2)}$. 
\end{prop}
\par
\begin{proof}
Recall that $BN^1_{3,(2)}\subset \barmoduli[{4,1}]$ parameterizes curves that admit a linear series $g^1_3$ with 
a section containing $2p$, where $p$ is the marked point. Suppose that $\bC$ intersects 
$BN^1_{3,(2)}$ at $(X, p)$. Let $\eta = \OO_X(3p)$ denote the theta characteristic given by $3p$.
Since $h^0(\eta)$ is odd,  Clifford's theorem implies that $h^0 (\eta) = 1$. 
Since $(X, p)$ is contained in $BN^1_{3,(2)}$, we have $h^0(\OO_X(2p+q)) = 2$ for
some $q$ different from $p$. By $\omega_X \sim \OO_X(6p)$ and Riemann-Roch,
we have $h^0(\OO_X(4p-q)) = 2$, hence $h^0(\OO_X(3p-q))\geq 1$. 
Note that $q$ is not a base point of the linear system $|\OO_X(3p)|$ for $q\neq p$. Consequently we have 
$h^0(\eta) =  1+ h^0(\OO_X(3p-q)) \geq 2$, contradicting that $h^0 (\eta) = 1$. 
\end{proof}
\par
\begin{proof}[{Proof of Theorem~\ref{thm:g4main}, Case ${\omoduli[4]}(6)^\odd$}]
Recall the divisor class of $BN^1_{3,(2)}$ in \eqref{BNp}. By $\bC\cdot BN^1_{3,(2)} = 0$, we have 
$$ \bC \cdot(4\omega_\rel + 8\lambda - \delta_0)=0.$$
It now suffices to use Proposition~\ref{prop:intomega} and to plug the
result in \eqref{eq:Lkappac} and \eqref{eq:Lcs}.
\end{proof}

\subsection{The stratum $\omoduli[4](5,1)$} \label{sec:51}
 
In the case  $\omoduli[4](5,1)$, we have
$$ L_{(5,1)} = 2 , \quad  s_{(5,1)} = 25/3, \quad c_{(5,1)} = 25/18. $$
\par
\begin{prop}
Let $C$ be a \teichm curve generated by $(X,\omega) \in \omoduli[4](5,1)$,
lifted to $\barmoduli[{4,1}]$ using the $5$-fold zero of $\omega$. Then $\bC$
does not intersect the pointed Brill-Noether divisor $BN^1_{3,(2)}$. 
\end{prop}
\par
\begin{proof}
Suppose that $(X, \omega)$ is contained in the intersection of $\bC$ 
with $BN^1_{3,(2)}$, where div$(\omega) = 5p+q$ with $p$ the marked point. 
By Proposition~\ref{prop:Teichhypdisj}, $X$ is not hyperelliptic.
For a non-hyperelliptic curve $X$ which is 
either smooth, or nodal irreducible, or consisting of two components 
joined at three nodes, its dual graph is three-connected, hence 
the dualizing sheaf $\omega_X$ 
is very ample due to Proposition~\ref{prop:dualgraph}. We will analyze the geometry of its canonical image in $\PP^3$. 
\par
We have $h^0(\OO_X(2p+r)) \geq 2$ for some smooth point $r$. Since $\omega_X \sim \OO_X(5p+q)$, 
by Riemann-Roch, it implies that $h^0(\OO_X(3p+q-r)) \geq 2$, 
hence $h^0(\OO_X(3p-r)) \geq 1$. If $r\neq p$, then $r$ is not a base point of 
$|\OO_X(3p)|$, hence $h^0(\OO_X(3p)) \geq 2$. If $r = p$, then we still have 
 $h^0(\OO_X(3p)) = h^0(\OO_X(2p+r)) \geq 2$. In any case, $3p$ admits a $g^1_3$ for $X$. 
By Riemann-Roch, $2p+q$ also yields a $g^1_3$. These must be two different $g^1_3$'s, 
for $p\neq q$. 
\par
Since $X$ is not hyperelliptic, its canonical image is contained in a quadric surface $Q$ 
in $\PP^3$. By Geometric Riemann-Roch, a section of a $g^1_3$ on $X$ corresponds to a 
line in $\PP^3$ that intersects $Q$ at three points (counting with multiplicity). By B\'{e}zout, this line must
 be a ruling of $Q$. Since $X$ has two $g^1_3$'s, $Q$ must be smooth and its two 
rulings correspond to the two $g^1_3$'s. 
But the two lines spanned by the sections $2p+q$ and $3p$ cannot be both tangent to $X$ at the smooth point $p$, contradiction. 
\end{proof}
\par
\begin{proof}[{Proof of Theorem~\ref{thm:g4main}, Case ${\omoduli[4]}(5,1)$}]
We lift the \teichm curve $\bC$ to $\barmoduli[{4,1}]$ 
using the $5$-fold zero of $\omega$. By the proposition $\bC\cdot BN^1_{3,(2)} = 0$, we have 
$$ \bC \cdot(4\omega_\rel + 8\lambda - \delta_0)=0.$$
By Proposition~\ref{prop:intomega}, we also have
$$ \bC\cdot \omega_\rel = \frac{\bC\ldotp \lambda - (\bC\ldotp \delta)/12}{11/3}. $$
Now the result follows by combining the two equalities. 
\end{proof}

\subsection{The stratum $\omoduli[4](4,2)^{\even*}$} \label{sec:42even}

In the case  $\omoduli[4](4,2)^\even$ we have
$$ L_{(4,2)^\even} = 32/15, \quad  s_{(4,2)^\even} = 17/2, \quad c_{(4,2)^\even} = 68/45. $$

Based on numerical values on individual \teichm curves, we believe that the sum of 
Lyapunov exponents is non-varying in this stratum. But we have not found a moduli space and a divisor 
to perform the desired disjointness argument. 

\subsection{The stratum $\omoduli[4](4,2)^{\odd*}$} \label{sec:42odd}

In the case  $\omoduli[4](4,2)^\odd$ we have
$$ L_{(4,2)^\odd} = 29/15 , \quad  s_{(4,2)^\odd} = 236/29, \quad c_{(4,2)^\odd} = 59/45. $$

We also believe that the sum of Lyapunov exponents is non-varying in this case. But we have not discovered a divisor 
that would do the job.

\subsection{The stratum $\omoduli[4](3,3)^\nh$} \label{sec:33nh}

In the case  $\omoduli[4](3,3)^\nh$ we have
$$ L_{(3,3)^\nh} = 2 , \quad  s_{(3,3)^\nh} = 33/4, \quad c_{(3,3)^\nh} = 11/8. $$
\par
\begin{prop}
Let $C$ be a \teichm curve generated by a flat surface 
$(X,\omega) \in \omoduli[4](3,3)^\nh$,
lifted to $\barmoduli[{4, 2}]$ (after a degree two base change). Then $\bC$
does not intersect the divisor $\Lin^1_3$. 
\end{prop}
\par
\begin{proof}
Recall that $\Lin^1_3\subset \barmoduli[{4, 2}]$ parameterizes pointed curves $(X, p, q)$ that admit a $g^1_3$ 
with a section vanishing at $p, q, r$ for some $r\in X$. Suppose $(X,p,q)$ is contained in the 
intersection of $\bC$ with $\Lin^1_3$. Since $\omega_X \sim \OO_X(3p+3q)$ and $h^0(\OO_X(p+q+r)) \geq 2$, by Riemann-Roch 
we know that $h^0(\OO_X(2p+2q-r)) \geq 2$. If $r\neq p, q$, then $h^0(\OO_X(2p+2q))\geq 3$, hence 
$2p+q$ and $2q+p$ both admit $g^1_3$. If $X$ is not hyperelliptic, using the canonical image of $X$ contained 
in a quadric in $\PP^3$ and the preceding argument of rulings, one concludes that $2p+q$ and $2q+p$ span the same line (connecting $p,q$) on the quadric, contradiction. If $r=p$ or $q$, again, $2p+q$ and $2q+p$ both admit
$g^1_3$ and consequently $X$ is hyperelliptic. But this stratum is non-hyperelliptic, and Proposition~\ref{prop:Teichhypdisj} yields the desired contradiction.
\end{proof}
\par
\begin{proof}[{Proof of Theorem~\ref{thm:g4main}, Case ${\omoduli[4]}(3,3)$}]
By $\bC \cdot \Lin^1_3 = 0$
together with Proposition~\ref{prop:intomega} and $\kappa_{(3,3)}=5/8$, the result follows immediately.
\end{proof}

\subsection{The stratum $\omoduli[4](2,2,2)^\odd$} \label{sec:222odd}

In the case  $\omoduli[4](2,2,2)^\odd$ we have
$$ L_{(2,2,2)^\odd} = 2, \quad  s_{(2,2,2)^\odd} = 8, \quad c_{(2,2,2)^\odd} = 4/3. $$
Note that by \cite[Proposition 7]{kz03}
this stratum contains the hyperelliptic curves where all
the zeros are fixed, but not those, where a pair of zeros are
exchanged.
\par
\begin{proof}[{Proof of Theorem~\ref{thm:g4main}, Case ${\omoduli[4]}(2,2,2)^\odd$}]
We just need to apply Proposition~\ref{prop:spinslope} to obtain the result.
It does apply to this case, because an odd theta characteristic on a genus four curve
cannot have three or more sections by Clifford's theorem (Theorem~\ref{thm:clifford}).
Note that such a theta characteristic $D$ is balanced, since 
$\deg(D|_Z) = w_Z/2 =  w_Z \deg(D)/(2g-2)$ for every component $Z$ of a stable curve
$X_\infty$ over the \teichm curve.
\end{proof}

\subsection{The stratum $\omoduli[4](3,2,1)$} \label{sec:321}

In the case  $\omoduli[4](3,2,1)$, we have
$$ L_{(3,2,1)} = 25/12, \quad  s_{(3,2,1)} = 41/5, \quad c_{(3,2,1)} = 205/144. $$
\par
\begin{prop}
Let $C$ be a \teichm curve generated by a flat surface 
$(X,\omega) \in \omoduli[4](3,2,1)$,
lifted to $\barmoduli[{4, 3}]$. Then $\bC$
does not intersect the divisor $BN^1_{4,(1,1,2)}$. 
\end{prop}

\begin{proof}
Recall that the Brill-Noether divisor $BN^1_{4,(1,1,2)}\subset
\barmoduli[4,3]$ parameterizes curves with a $g^1_4$ given by $p+q+2r$. 
Suppose that $(X, \omega)$ is contained in the intersection of $\bC$ with $BN^1_{4, (1,1,2)}$, 
where div($\omega) = 3p + 2q + r$ and $p,q,r$ are the (ordered) marked points. 
By $h^0(\OO_X(p+q+2r)) \geq 2$ and Riemann-Roch, we have
$h^0(\OO_X(2p+q- r)) \geq 1$. 
Consequently $h^0(\OO_X(2p+q))\geq 2$ and by Riemann-Roch again, we have
$h^0(\OO_X(p+q+r)) \geq 2$. 
Note that if $2p+q \sim p+q+r$, then $p\sim r$, which is impossible. 
If these are two different $g^1_3$'s and $X$ is smooth, or
stable but non-hyperelliptic and at least three-connected, then its canonical map is an
embedding. Consequently both $p$ and $q$ lie on two different rulings of
the quadric containing the canonical image of $X$. This is impossible.
\par
Since hyperelliptic stable fibers cannot occur by Proposition~\ref{prop:Teichhypdisj}, 
the last case to be excluded consists of a stable curve which is only two-connected.
Given the constraints in Corollary~\ref{cor:zerosatinf}, there are two possible types 
for such a stable curve. First, there are two irreducible nodal 
curves $X_1$ and $X_2$ of arithmetic genus one and two, respectively, joined
at two nodes $\{x,y\}$ with $q$ lying on $X_1$ and $p$ and $r$ lying on $X_2$.
Second, there are irreducible nodal curves $X_1$, $X_0$ and $Y_1$, 
the index specifying the arithmetic genus with $p$ on $Y_1$, $q$ on $X_1$
and $r$ on $X_0$, whose intersection is given by $Y_1 \cdot X_1 = \{x\}$, $Y_1 \cdot X_0 = \{z_1,z_2\}$,
$X_1 \cdot X_0 = \{y\}$.
\par
For the first type, consider the linear system $|\OO_X(2p+q)|$. Since $p$ and $q$ lie on different
components of the stable curve, $q$ has to be a base point of this linear system.
Hence $\omega_{X_2}(x+y) \sim \cO_{X_2}(4p)$, i.e.\ $p$
is a Weierstrass point
for the line bundle $\omega_{X_2}(x+y)$. By the same argument $q$ is a
base point of $|\OO_X(p+q+r)|$, hence $\omega_{X_2}(x+y) \sim \cO_{X_2}(2p+2r)$.
Since $p \neq r$, this implies that $p$ is not a Weierstrass point and
we obtain the desired contradiction.
\par
For the second type, the condition $h^0(\OO_X(p+q+r)) \geq 2$ provides
an immediate contradiction, since all the three points lie on different
components of the stable curve.
\end{proof}

\begin{proof}[{Proof of Theorem~\ref{thm:g4main}, Case ${\omoduli[4]}(3,2,1)$}]
Recall the divisor class of $BN^1_{(1,1,2)}$ in \eqref{BN(1,1,2)}. 
By $\bC \cdot BN^1_{4, (1,1,2)} = 0$ together with Proposition~\ref{prop:intomega}, the result 
follows directly. 
\end{proof}

\subsection{Varying sum in the stratum $\omoduli[4](2,2,2)^\even$}

In the case  $\omoduli[4](2,2,2)^\even$, we have 
$$ L_{(2,2,2)^\even} = 166/75, \quad  s_{(2,2,2)^\even} = 696/83, \quad c_{(2,2,2)^\even} = 116/75. $$
Note that by \cite[Proposition 7]{kz03}
this stratum contains the hyperelliptic curves where 
a pair of zeros are exchanged, but not those, where
all the zeros are fixed. For \teichm curves $C^\hyp$ contained in the locus of hyperelliptic flat surfaces  
within this stratum, we have 
$$L(C^\hyp) = 7/3. $$
\par
\begin{exs}
The square-tiled surface $(X: y^6=x(x-1)(x-t), \omega = dx/y)$ found by Forni and Matheus
\cite{fornimatheus} has maximally degenerate Lyapunov spectrum, i.e.\ $L(C) = 1.$ 
\end{exs}
\par
\begin{prop} 
A \teichm curve $C$ generated by a non-hyperelliptic flat surface 
$(X,\omega) \in \obarmoduli[4](2,2,2)^\even$
has 
$$ L(C) \leq 16/7.$$
\par
In particular the sum of Lyapunov exponents of any \teichm curve 
generated by a non-hyperelliptic flat surface in this stratum is strictly
smaller than the sum of Lyapunov exponents of any \teichm curve 
generated by a hyperelliptic flat surface  in this stratum. 
\end{prop}
\par
\begin{proof}
Recall the divisor class in \eqref{eq:GP} of the Gieseker-Petri divisor $GP$ on $\barmoduli[4]$. 
It has slope equal to $17/2$. Hence if $C$ is
not entirely contained in this divisor, we have $s(C) \leq 17/2$ by Lemma~\ref{lem:upperbound}, 
which translates into $ L(C) \leq 16/7$.
\par
If $C$ is contained in $GP$, we try to intersect $C$ with the
Brill-Noether divisor $BN^1_{4,(3,1)}$. If $C$ is not contained
in  $BN^1_{4,(3,1)}$, using Proposition~\ref{prop:intomega} with $\kappa = 2/3$ we obtain 
a better bound $L(C) \leq 2$.
\par
Suppose that a generic surface $(X,\omega)$ parameterized by $C$ lies 
in $BN^1_{4,(3,1)}$. By definition, we have $h^0(\OO_X(3p+q)) \geq 2$, which
implies $h^0(\OO_X(q+2r-p)) \geq 1$ by Riemann-Roch.
Consequently we have $h^0(\OO_X(2r+q)) \geq 2$ and $h^0(\OO_X(2p + q)) \geq 2$.
Since we excluded hyperelliptic $X$, the assumption that $X$ is parameterized in
$GP$ implies that $2r + q \sim 2p+q$. Hence $2p \sim 2r$, which contradicts 
the assumption that $X$ is not hyperelliptic.
\end{proof}

\subsection{Varying sum in the stratum $\omoduli[4](4,1,1)$}
                       
In this stratum we have
$$ L_{(4,1,1)} = 1137/550  \approx 2.06727, \quad  s_{(4,1,1)} = 3118/379, 
\quad c_{(4,1,1)} = 1559/1100.$$
The stratum contains the locus of hyperelliptic flat surfaces coming from $\cQ(3,2,-1^{9})$. Hence for a \teichm curve $C^\hyp$
in this locus, the sum of Lyapunov exponents is 
$$L(C^\hyp) = 23/10.$$
\par
\begin{exs} {\rm
A \teichm curve $C$ generated by
the square-tiled surface with  
\begin{equation} \label{eq:411}
(\pi_r = (12)(3)(4)(5)(6\ 7)(8)(9\ 10), \pi_u = (132456879)(10))
\end{equation}
has $L(C)  = 1043/530 \approx 1.96792.$
\par
A \teichm curve $C$ generated by
the square-tiled surface with  
$$(\pi_r = (12)(3)(4)(5)(6)(7,8)(9)(10\ 11), \pi_u = (1\,3)(2,4,5,6,7,9)(8\,10)(11))$$
has $L(C)  = 267163/129510 \approx 2.06287.$
\par
This stratum also contains \teichm curves $C$ generated by
square-tiled surfaces with
$$L(C) \in \{1043/530 \approx 1.96792, 579/290, 4101/1990,
1799/870 \approx 2.06782, 23/10  \}.$$
}\end{exs}
\par
\begin{prop}
A \teichm curve $C$ generated by a non-hyperelliptic flat surface 
$(X,\omega) \in \omoduli[4](4,1,1)$
has 
$$ L(C) \leq 21/10.$$
\par
In particular the sum of Lyapunov exponents of any \teichm curve 
generated by a non-hyperelliptic flat surface in this stratum is strictly
smaller than the sum of Lyapunov exponents of any \teichm curve 
generated by a hyperelliptic flat surface in this stratum. 
\end{prop}
\par
\begin{proof}
By the same argument as in the proof for the stratum $\omoduli[3](2,1,1)$
we may pass to an unramified covering and label the zeros of $\omega$ such 
that $\omega_X \sim \OO_X(4p+q+r)$ along the whole family. Using $p$ and $q$ 
we lift this covering
to a curve in $\barmoduli[{4,2}]$ that we continue to call $C$.
We will show that $C$ is not entirely contained in the divisor
$BN^1_{4,(2,2)}$, provided that $X$ is not hyperelliptic.
Then $\bC\cdot BN^1_{4,(2,2)} \geq 0$ together with 
Proposition~\ref{prop:intomega} using $\kappa_{(4,1,1)}=39/60$ implies the claim.
\par
Suppose a generic flat surface $(X,\omega)$ in $C$ is contained in $BN^1_{4,(2,2)}$.
Then by definition $h^0(\OO_X(2p+2q)) \geq 2$ and by Riemann-Roch 
$h^0(\OO_X(2p-q+r)) \geq 1$. Hence $h^0(\OO_X(2p+r)) \geq 2$ and by Riemann-Roch again
$h^0(\OO_X(2p+q)) \geq 2$. Since $X$ is not hyperelliptic, we consider
the quadric surface containing its canonical image. The ruling that is
tangent to $X$ at $p$ intersects $X$ at a third point. This point has to
be both $q$ and $r$ due to the $g^1_3$'s given by $2p+q$ and $2p+r$,
which is absurd for $q \neq r$. 
\end{proof}
\par
%
\subsection{Varying sum in the stratum $\omoduli[4](3,1,1,1)$}

In this stratum we have
$$ L_{(3,1,1,1)} =  66/31 \approx 2.12903, \quad  
s_{(3,1,1,1)} = 65/8, \quad c_{(3,1,1,1)} = 715/496.$$
This stratum does not contain any submanifolds obtained by double covering constructions.
\par
\begin{exs} {\rm
The \teichm curve $C$ generated by the square-tiled surface with  
\begin{equation} \label{eq:3111}
(\pi_r = (123456789\,10), \pi_u = (145836\,10)(279))
\end{equation}
has $L(C)  = 514/243  \approx 2.11523.$
\par
The \teichm curve $C$ generated by the square-tiled surface with 
$$(\pi_r = (123456789\,10\,11), \pi_u = (1458\,10\,27\,11)(369))$$ 
has $L(C)  = 1531/720  \approx 2.12639.$
\par
There exist \teichm curves $C$ generated by square-tiled surfaces in this stratum with 
$$ L(C) \in \{241/114 \approx 2.114035,  72167/33984, 1531/720 \approx 2.1263\}.$$
}
\end{exs}
\par
\begin{prop}
A \teichm curve $C$ generated by a flat surface 
$(X,\omega) \in \obarmoduli[4](3,1,1,1)$
has 
$$ L(C) \leq 7/3.$$
\end{prop}
\par
\begin{proof}
The proof is identical to the one given below for the stratum 
$\omoduli[4](2,1,1,1,1)$ using two different lifts
to $\barmoduli[{4,3}]$ and the divisor $BN^1_{4,(1,1,2)}$ (see Section~\ref{sec:21111}). 
\end{proof}

\subsection{Varying sum in the stratum $\omoduli[4](2,2,1,1)$}

In this stratum we have
$$ L_{(2,2,1,1)} =  5045/2358 \approx 2.13952, \quad  s_{(2,2,1,1)} = 8178/1009, 
\quad c_{(2,2,1,1)} = 6815/4716.$$
\par 
The stratum contains two  loci of hyperelliptic flat surfaces. One of them 
corresponds to the orientation double covers of $\cQ(4,2,-1^{10})$, 
hence for a \teichm curve $C$
in this locus, the sum of Lyapunov exponents is $L(C) = 5/2$. In this locus,
the zeros are permuted in pairs by the hyperelliptic involution.
\par
The second one corresponds to $\cQ(2,1,1,-1^8)$, hence for a \teichm curve $C$
in this locus, the sum of Lyapunov exponents is $L(C) = 13/6 \approx 2.16$.
\par
\begin{exs}{\rm
The \teichm curve $C$ generated by the
square-tiled surface with $$(\pi_r = (12)(3)(4)(5)(67)(8)(9)(10\,11)(12), \\ 
\pi_u = (134)(256789\,10\,11\,12))$$
has $L(C)  = 3313/1590\approx 2.083.$
\par
The \teichm curve $C$ generated by the
square-tiled surface with $$(\pi_r = (12)(3)(4)(56)(7)(89)(10\,11), \\
\pi_u = (134578)(269\,10)(11))$$ has 
$L(C)  = 4919/2312 \approx 2.1275.$
\par
There exist \teichm curves $C$ generated by 
square-tiled surfaces in this stratum with
$$ L(C) \in \{3313/1590\approx  2.083,157/75, 273529/128580, 4919/2312\approx  2.1275 \}.$$
}
\end{exs}
\begin{prop}
A \teichm curve $C$ generated by a non-hyperelliptic flat surface 
$(X,\omega) \in \obarmoduli[4](2,2,1,1)$
has 
$$ L(C) \leq 13/6.$$
\par
In particular the sum of Lyapunov exponents of any \teichm curve 
generated by a non-hyperelliptic flat surface in this stratum is strictly
smaller than the sum of Lyapunov exponents of any \teichm curve 
generated by a hyperelliptic flat surface where the four zeros are permuted in pairs.
\end{prop}
\par
\begin{proof}
By the same argument as in the proof for the stratum $\omoduli[3](2,1,1)$
we may pass to an unramified covering and label the zeros of $\omega$ such 
that $\omega_X \sim \OO_X(2p+2q+r+s)$ along the whole family. 
We lift this covering to a curve in $\barmoduli[{4,3}]$ by marking $p,q,r$
and continue to call it $C$. We will show that $C$ is not entirely contained in the divisor
$BN^1_{4,(1,1,2)}$. Then $\bC\cdot BN^1_{4,(1,1,2)}\geq 0$ together with Proposition~\ref{prop:intomega} 
implies this proposition.
\par
Suppose a generic flat surface $(X,\omega)$ parameterized by $C$ is contained in $BN^1_{4,(1,1,2)}$.
Then by definition $h^0(\OO_X(p+q+2r)) \geq 2$ and by Riemann-Roch 
$h^0(\OO_X(p+q-r+s)) \geq 1$. Hence $h^0(\OO_X(p+q+s)) \geq 2$ and by Riemann-Roch again
$h^0(\OO_X(p+q+r)) \geq 2$. For $X$ non-hyperelliptic this means that on the
quadric containing the canonical image of $X$ in $\PP^3$, there are two rulings passing through
$p$ and $q$ (hence they are the same ruling), one intersecting the curve moreover at $r$ and the other 
intersecting the curve moreover at $s$. This is impossible for $r\neq s$.
\end{proof}
\par

\subsection{Varying sum in the stratum $\omoduli[4](2,1,1,1,1)$}\label{sec:21111}

In this stratum we have
$$ L_{(2,1,1,1,1)} = 131/60  \approx 2.18333, \quad  s_{(2,1,1,1,1)} = 1052/131, \quad 
c_{(2,1,1,1,1)} = 263/180.$$

This stratum contains the locus of hyperelliptic flat surfaces
corresponding to $\cQ(3,2,2,-1^{11})$. Hence for a \teichm curve $C$
in this locus, the sum of Lyapunov exponents is $L(C) = 14/5 = 2.8$.
\par
\begin{exs}{\rm
The \teichm curve $C$ generated by the
square-tiled surface with $$(\pi_r = (12)(3)(4)(5,6)(7)(8)(9\,10)(11\,12)(13), \\ 
\pi_u = (132457689\,11\,10\,12\,13))$$
has $L(C)  = 268/129 \approx 2.0775.$
\par
The \teichm curve $C$ generated by the
square-tiled surface with $$(\pi_r = (12)(3)(4)(5)(6)(78)(9\,10)(11)(12\,13), \\
\pi_u = (13)(245679\,11\,8\,12\,10)(13))$$ has 
$L(C)  = 207826/95511 \approx  2.1759.$
\par
There exist \teichm curves $C$ generated by 
square-tiled surfaces in this stratum with 
$$ L(C) \in \{268/129 \approx 2.0775, 239/114, 4031/1923, 207826/95511 \approx  2.175 \}.$$
}\end{exs}
\par
\begin{prop}
A \teichm curve $C$ generated by a non-hyperelliptic flat surface 
$(X,\omega) \in \obarmoduli[4](2,1,1,1,1)$
has 
$$ L(C) \leq 7/3.$$
\par
In particular the sum of Lyapunov exponents of any \teichm curve 
generated by a non-hyperelliptic flat surface in this stratum is strictly
smaller than the sum of Lyapunov exponents of any \teichm curve 
generated by a hyperelliptic flat surface  in this stratum. 
\end{prop}
\par
\begin{proof}
By the same argument as in the proof for the stratum $\omoduli[3](2,1,1)$, 
we may pass to an unramified covering and label the zeros of $\omega$ such 
that $\omega_X \sim \OO_X(2p+q+r+s+u)$ along the whole family. First, using 
$p$, $q$ and $r$ we lift this covering
to a curve in $\barmoduli[{4,3}]$ that we continue to call $C$.
\par
If $C$ is not entirely contained in the divisor
$BN^1_{4,(1,2,1)}$, then $\bC\cdot BN^1_{4,(1,2,1)} \geq 0$ together with 
Proposition~\ref{prop:intomega} using $\kappa_{(2,1,1,1,1)}=13/18$ implies the claim.
If $C$ is entirely contained in the divisor $BN^1_{4,(1,2,1)}$, 
we can lift $C$ to $\barmoduli[{4,3}]$ alternatively by marking $p$, $q$ and $r$.
Again, if $C$ is not contained in $BN^1_{4,(1,2,1)}$, the claim holds.
\par
Suppose that $C$ is contained in the Brill-Noether divisor for
both lifts. Then for $(X,\omega)$ parameterized in $C$, 
by definition we have $h^0(\OO_X(p+2q+r)) \geq 2$, consequently
we obtain $h^0(\OO_X(s+u+p-q)) \geq 1$ and $h^0(\OO_X(s+u+p)) \geq 2$. For
the second lift we deduce from  $h^0(\OO_X(p+2q+u)) \geq 2$
that $h^0(\OO_X(s+r+p)) \geq 2$. Since $X$ is not hyperelliptic, 
the canonical map is an embedding and its image lies on a quadric surface in $\PP^3$.
Then the unique line on the quadric passing through $s$ and $p$ cannot
have a third intersection point with $C$ at both $r$ and $u$ for $r\neq u$.
\end{proof}
\par

\subsection{Varying sum in the stratum $\omoduli[4](1,1,1,1,1,1)$}

In this stratum we have
$$ L_{(1,1,1,1,1,1)} =  \frac{839}{377} \approx 2.22546, \quad  s_{(1,1,1,1,1,1)} = 
\frac{6675}{839}, 
\quad c_{(1,1,1,1,1,1)} = \frac{2225}{1508}.$$
The stratum contains the locus of hyperelliptic flat surfaces 
corresponding to the stratum $\cQ(2,2,2,-1^{10})$. Hence for a \teichm curve $C$
in this locus we have $L(C) = 5/2$.
\par
\begin{exs}{\rm
The \teichm curve $C$ generated by the
square-tiled surface with $$(\pi_r = (12)(3)(4)(5)(6)(78)(9)(10)(11\,12)(13)(14), \\ 
\pi_u = (132456798\,10\,11\,13)(12\,14) )$$
has $L(C)  = 125/58 \approx 2.15517$.
\par
The \teichm curve $C$ generated by the
square-tiled surface with $$(\pi_r = (12)(3)(4)(56)(7)(8)(9\,10)(11)(12), \\
\pi_u = (132457689\,11\,10\,12))$$ has 
$L(C)= 9/4=2.25$. 
\par
There exist \teichm curves $C$ generated by  
square-tiled surfaces in this stratum with
$$ L(C) \in \{125/58 \approx 2.15517, 419/194,  1019/470, 8498/3867 \approx 2.1975, 9/4 \}.$$
}\end{exs}
\par
\begin{prop}
A \teichm curve $C$ generated by a flat surface 
$(X,\omega) \in \obarmoduli[4](1,1,1,1,1,1)$
has 
$$ L(C) \leq 5/2.$$
\end{prop}
\par
\begin{proof}
The argument is completely analogous to the stratum $\obarmoduli[4](2,1,1,1,1)$.
\end{proof}
\par

\section{Genus five}\label{sec:g=5}

In genus $5$ only few strata have a non-varying sum of Lyapunov exponents.
We summarize the results in Tables~\ref{cap:5part1} and \ref{cap:5part2}. Contrary to genus $4$ we do
not give an upper bound for the sum in all the (components of) strata where
the sum is varying but provide only one example which often comes from the locus of
hyperelliptic flat surfaces.

\begin{figure}
    \centering
      
$$
\begin{array}{|c|c||c|c|c|c|c|}

\hline
&&\multicolumn{5}{|c|}{}\\

\multicolumn{1}{|c|}{\text{Degrees}}&
\multicolumn{1}{|c||}{\text{Hyperelliptic}}&
\multicolumn{5}{|c|}{\text{Lyapunov exponents}}\\

\multicolumn{1}{|c|}{\text{of }}&
\multicolumn{1}{|c||}{\text{or spin}}&
\multicolumn{5}{|c|}{\text{}}\\

\cline{3-7}

\multicolumn{1}{|c|}{\text{zeros}}&
\multicolumn{1}{|c||}{\text{structure}}&
\multicolumn{2}{|c|}{\text{}}&
\multicolumn{3}{|c|}{\text{}}\\

\multicolumn{1}{|c|}{(d_1,\dots,d_n)}&
\multicolumn{1}{|c||}{}&
\multicolumn{2}{|c|}{\text{Component}}&
\multicolumn{3}{|c|}{\text{\teichm curves}}\\

\multicolumn{1}{|c|}{\text{}}&
\multicolumn{1}{|c||}{\text{}}&
\multicolumn{2}{|c|}{\text{}}&
\multicolumn{3}{|c|}{\text{}}\\

\cline{3-7}

\multicolumn{1}{|c|}{\text{}}&
\multicolumn{1}{|c||}{\text{}}&
\multicolumn{1}{|c|}{\approx}&
\multicolumn{1}{|c|}{\overset{g}{\underset{j=1}{\sum}} \lambda_j} &
\multicolumn{1}{|c|}{\approx}&
\multicolumn{1}{|c|}{\overset{g}{\underset{j=1}{\sum}} \lambda_j}&
\multicolumn{1}{|c|}{\text{Reference}}\\
[-\halfbls] &&&&&&\\
\hline &&&& \multicolumn{2}{|c|}{} & \\  [-\halfbls]
(8) & \text{hyperelliptic} & 2.777778 & \frac{25}{9} & \multicolumn{2}{c}{\text{Non-varying}} & \text{Thm.~\ref{thm:calchyp}}\\
[-\halfbls] &&&&\multicolumn{2}{|c|}{}&\\
\hline &&&& \multicolumn{2}{|c|}{} & \\  [-\halfbls]
(8) & \text{even} & 2.222222 & \frac{20}{9} & \multicolumn{2}{|c|}{\text{Non-varying}} & \text{Sec.~\ref{sec:8even}}\\
[-\halfbls] &&&&\multicolumn{2}{|c|}{}&\\
\hline &&&& \multicolumn{2}{|c|}{} & \\  [-\halfbls]
(8) & \text{odd} & 2.111111 & \frac{19}{9}  & \multicolumn{2}{|c|}{\text{Non-varying}} & \text{Sec.~\ref{sec:8odd}}\\
[-\halfbls] &&&&\multicolumn{2}{|c|}{}&\\
\hline &&&&&& \\ [-\halfbls]
(7, 1) & - & 2.227022 & \frac{2423}{1088} & 2.229062 &\frac{7133}{3200} & \text{Eq.~\eqref{eq:ST71}}\\
[-\halfbls] &&&&&&\\
\hline &&&&&& \\ [-\halfbls]
(6, 2) & \text{even} &  2.301983 & \frac{178429}{77511} & 2.619047 & \frac{55}{21} & \cQ({5, 1},-1^{10}) \\
[-\halfbls] &&&&&&\\
\hline &&&& \multicolumn{2}{|c|}{} & \\  [-\halfbls]
(6, 2) & \text{odd} &  2.190476 & \frac{46}{21} & \multicolumn{2}{|c|}{\text{Non-varying\ ?}} & - \\
[-\halfbls] &&&&\multicolumn{2}{|c|}{}&\\
\hline &&&&&& \\ [-\halfbls]
(6, 1, 1) & - & 2.285384 & \frac{59332837}{25961866} & 2.785714 & \frac{39}{14} & \cQ({5, 2},-1^{11}) \\
[-\halfbls] &&&&&&\\
\hline &&&& \multicolumn{2}{|c|}{} & \\  [-\halfbls]
(5, 3) & -   & 2.250000 & \frac{9}{4}  & \multicolumn{2}{|c|}{\text{Non-varying}} & \text{Sec.~\ref{sec:53}} \\
[-\halfbls] &&&&\multicolumn{2}{|c|}{}&\\
\hline &&&&&& \\ [-\halfbls]
(5, 2, 1) & -  & 2.300563 & \frac{4493}{1953} & 2.302594 & \frac{48541}{21081} & \text{Eq.\ \eqref{eq:ST521}} \\
[-\halfbls] &&&&&&\\
\hline &&&&&& \\ [-\halfbls]
(5, 1, 1, 1) & - &  2.340909 & \frac{103}{44}  &  2.337802 & \frac{12381}{5296} & \text{Eq.\ \eqref{eq:ST5111}} \\
[-\halfbls] &&&&&&\\
\hline &&&& \multicolumn{2}{|c|}{} & \\  [-\halfbls]
(4, 4) & \text{hyperelliptic} & 3.000000 & 3 & \multicolumn{2}{c}{\text{Non-varying}} & \text{Thm.~\ref{thm:calchyp}} \\
[-\halfbls] &&&&\multicolumn{2}{|c|}{}&\\
\hline &&&&&& \\ [-\halfbls]
(4, 4) & \text{even} & 2.311111 & \frac{104}{45} & 2.400000 & \frac{12}{5} & \text{Eq.\ \eqref{eq:ST44even}} \\
[-\halfbls] &&&&&&\\
\hline &&&&&& \\ [-\halfbls]
(4, 4) & \text{odd} & 2.191613 & \frac{228605}{104309} & 2.600000& \frac{13}{5} & \cQ({3, 3},-1^{10})\\
[-\halfbls] &&&&&&\\
\hline &&&&&& \\ [-\halfbls]
(4, 3, 1) & - & 2.306255 & \frac{438419}{190100} & 2.302715 & \frac{777627}{337700} & \text{Eq.\ \eqref{eq:ST431}}\\
[-\halfbls] &&&&&&\\
\hline &&&&&& \\ [-\halfbls]
(4, 2, 2) & \text{even} & 2.374007 & \frac{34981}{14735} & 2.800000 & \frac{14}{5} & \cQ({4, 3},-1^{11})\\
[-\halfbls] &&&&&&\\
\hline &&&&&& \\ [-\halfbls]
(4, 2, 2) & \text{odd}  & 2.260315 & \frac{538102}{238065} & 2.466666 &  \frac{37}{15} & \cQ({3, 1, 1},-1^9) \\
[-\halfbls] &&&&&&\\
\hline &&&&&& \\ [-\halfbls]
(4, 2, 1, 1) & - & 2.354799 & \frac{646039}{274350} & 2.633333 & \frac{79}{30} & \cQ({3, 2, 1},-1^{10})\\
[-\halfbls] &&&&&&\\ \hline
\end{array}
$$
\caption{Varying and non-varying sums in genus five, part I} \label{cap:5part1}
\end{figure}
\par

\begin{figure}
    \centering      
$$
\begin{array}{|c|c||c|c|c|c|c|}

\hline
&&\multicolumn{5}{|c|}{}\\

\multicolumn{1}{|c|}{\text{Degrees}}&
\multicolumn{1}{|c||}{\text{Hyperelliptic}}&
\multicolumn{5}{|c|}{\text{Lyapunov exponents}}\\

\multicolumn{1}{|c|}{\text{of }}&
\multicolumn{1}{|c||}{\text{or spin}}&
\multicolumn{5}{|c|}{\text{}}\\

\cline{3-7}

\multicolumn{1}{|c|}{\text{zeros}}&
\multicolumn{1}{|c||}{\text{structure}}&
\multicolumn{2}{|c|}{\text{}}&
\multicolumn{3}{|c|}{\text{}}\\

\multicolumn{1}{|c|}{(d_1,\dots,d_n)}&
\multicolumn{1}{|c||}{}&
\multicolumn{2}{|c|}{\text{Component}}&
\multicolumn{3}{|c|}{\text{Teichm\"uller curves}}\\

\multicolumn{1}{|c|}{\text{}}&
\multicolumn{1}{|c||}{\text{}}&
\multicolumn{2}{|c|}{\text{}}&
\multicolumn{3}{|c|}{\text{}}\\

\cline{3-7}

\multicolumn{1}{|c|}{\text{}}&
\multicolumn{1}{|c||}{\text{}}&
\multicolumn{1}{|c|}{\approx}&
\multicolumn{1}{|c|}{\overset{g}{\underset{j=1}{\sum}} \lambda_j} &
\multicolumn{1}{|c|}{\approx}&
\multicolumn{1}{|c|}{\overset{g}{\underset{j=1}{\sum}} \lambda_j}&
\multicolumn{1}{|c|}{\text{Reference}}\\
[-\halfbls] &&&&&& \\
\hline &&&&&& \\ [-\halfbls]
(4, 1^4) & -  & 2.393586 & \frac{640763}{267700} & 2.800000 & \frac{14}{5} & \cQ({3, 2, 2},-1^{11})\\
[-\halfbls] &&&&&&\\
\hline &&&&&& \\ [-\halfbls]
(3, 3, 2) & - &  2.318020 & \frac{61307}{26448} & 2.833333 & \frac{17}{6} & \cQ({6, 1},-1^{11}) \\
[-\halfbls] &&&&&&\\
\hline &&&&&& \\ [-\halfbls]
(3, 3, 1, 1) & -  & 2.358542 & \frac{47435}{20112} & 3.000000 & 3 & \cQ({6, 2},-1^{12}) \\
[-\halfbls] &&&&&&\\
\hline &&&&&& \\ [-\halfbls]
(3, 2, 2, 1) & - & 2.366588 & \frac{6049}{2556} &  2.362268 & \frac{2041}{864} & \text{Eq.\ \eqref{eq:ST3221}} \\
[-\halfbls] &&&&&&\\
\hline &&&&&& \\ [-\halfbls]
(3, 2, 1^3) & - & 2.405498 & \frac{700}{291} & 2.398764 & \frac{3495}{1457} & \text{Eq.\ \eqref{eq:ST32111}} \\
[-\halfbls] &&&&&&\\
\hline &&&&&& \\ [-\halfbls]
(3, 1^5) & - & 2.443023 & \frac{2101}{860}  & 2.431085 & \frac{77785}{31996} & \text{Eq.\ \eqref{eq:ST31to5}} \\
[-\halfbls] &&&&&&\\
\hline &&&&&& \\ [-\halfbls]
(2, 2, 2, 2) & \text{even} & 2.434379 & \frac{2096}{861} & 2.666666 & \frac{8}{3} & \cQ({4, 1, 1},-1^{10}) \\
[-\halfbls] &&&&&&\\
\hline &&&&&& \\ [-\halfbls]
(2, 2, 2, 2) & \text{odd} &  2.319961 & \frac{355309}{153153} & 2.333333 & \frac{7}{3} & \cQ({1, 1, 1, 1},-1^8) \\
[-\halfbls] &&&&&&\\
\hline &&&&&& \\ [-\halfbls]
(2^3, 1^2) & - &  2.413574 & \frac{79981}{33138} &2.833333 & \frac{17}{6} & \cQ({4, 2, 1},-1^{11}) \\
[-\halfbls] &&&&&&\\
\hline &&&&&& \\ [-\halfbls]
(2, 2, 1^4) & - & 2.451217 & \frac{266761}{108828} & 2.666666 & \frac{8}{3} & \cQ({2, 2, 1, 1},-1^{10})\\
[-\halfbls] &&&&&&\\
\hline &&&&&& \\ [-\halfbls]
(2, 1^6) & - &  2.487756 & \frac{35861}{14415} &2.833333 & \frac{17}{6} & \cQ({2, 2, 2, 1},-1^{11})\\
[-\halfbls] &&&&&&\\
\hline &&&&&& \\ [-\halfbls]
(1^8) & - & 2.523451 & \frac{235761}{93428} & 3.000000 & 3 & \cQ({2, 2, 2, 2},-1^{12})\\
[-\halfbls] &&&&&&\\ \hline
\end{array}
$$
\caption{Varying and non-varying sums in genus five, part II} \label{cap:5part2}
\end{figure}
\par
\medskip

\subsection{The stratum $\omoduli[5](8)^\even$} \label{sec:8even}

In the case  $\omoduli[5](8)^\even$ we have
$$ L_{(8)^\even} = 20/9 , \quad  s_{(8)^\even} = 8, \quad c_{(8)^\even} = 50/27. $$

\begin{prop}
Let $C$ be a \teichm curve generated by a flat surface in $\omoduli[5](8)^\even$. Then 
$\bC$ does not intersect the Brill-Noether divisor $BN^1_3$ on $\barmoduli[5]$. 
\end{prop}
\par
\begin{proof}
\teichm curves in this stratum are disjoint from the hyperelliptic locus even at the boundary of
$\barmoduli[5]$, since the hyperelliptic component is a different component
and by Proposition~\ref{prop:Teichhypdisj}. Suppose $(X,\omega)$ is a flat surface contained in the intersection of $\bC$
and $BN^1_3$. Since $X$ is trigonal (possibly nodal but irreducible) and is not hyperelliptic, 
its canonical image lies on a cubic scroll surface in $\PP^4$ whose rulings are spanned by 
the sections of the $g^1_3$ (see e.g. \cite[Section 2.10]{reid}). This scroll surface can be either smooth or singular, corresponding to Hirzebruch surfaces $F_n$ of two types, respectively (see \cite[Chapter IV]{beauville} or \cite[Section 2]{coskun} for preliminaries on Hirzebruch surfaces). Here we follow the notation in \cite{coskun}. 
\par
Suppose the scroll surface is smooth as the embedding of the Hirzebruch 
surface $F_1$ by the linear system $| e+ 2f |$, where 
$$e^2 = -1, \quad e\cdot f = 1, \quad f^2 = 0.$$ 
Then $X$ has class $3e + 5f$. Note that $4p$ admits a $g^1_4$, 
which comes from the projection of $X$ from a plane $\Lambda$ 
to a line in $\PP^4$. This plane $\Lambda$ 
intersects $X$ at $\geq 4$ points (with multiplicity) and the 
intersection contains the residual $4p$. But $F_1$ has degree three, 
so the intersection $F_1 \cap \Lambda$ consists of a curve 
$B$ with possibly finitely many points outside $B$. If $B$ is 
a ruling, then $B\cdot X = 3$ and $\Lambda$ also intersects $X$ at 
a point outside $B$, such that $\Lambda$ is spanned by 
$B$ and that point. Then we cannot have $4p \subset \Lambda \cap X$, 
contradiction. If $B$ has higher degree, it can only be a conic 
(or its degeneration) of class $e+f$. Then $B\cdot X = 5$, so 
$\Lambda\cap X = 4p + q$ admits a $g^2_5$ by 
Geometric Riemann-Roch. For $q\neq p$ we obtain $h^0(\OO_X(4p-q)) = 2$ 
and hence $h^0(\OO_X(4p)) = 3$, contradiction. For $q = p$, the residual $3p$ 
admits a $g^1_3$, so it gives rise to a ruling $L$ on the cubic scroll. Then $L$ and $B$ are both
tangent to $X$ at $p$. But $L\cdot B = f\cdot (e+f) = 1$, leading to a contradiction. 
\par
If the scroll is singular, it is isomorphic to $F_3$ by the linear system $|e+3f|$, where 
$$e^2 = -3, \quad e\cdot f = 1, \quad f^2 = 0.$$
Since $X\cdot f = 3$ and $X\cdot (e+3f) = 8$, it has class $3e + 8f$. 
Then $X\cdot e = -1$, which implies that $X$ consists of
$e$ union a curve of class $2e+8f$, contradicting the irreducibility of $X$. 
\end{proof}
\par
\begin{proof} [{Proof of Theorem~\ref{thm:g5main}, Case ${\omoduli[5]}(8)^\even$}]
By the proposition we have $\bC\cdot BN^1_3  = 0$. Since this divisor has slope equal to $8$ by~\eqref{BN}, the \teichm
curve $C$ has the same slope $s(C) = 8$.
%
\end{proof}

\subsection{The stratum $\omoduli[5](8)^\odd$} \label{sec:8odd}

In the case  $\omoduli[5](8)^\odd$ we have
$$ L_{(8)^\odd} = 19/9, \quad  s_{(8)^\odd} = 148/19, \quad c_{(8)^\odd} = 37/27. $$

\begin{prop}
Let $C$ be a \teichm curve generated by a flat surface $(X,\omega) \in \omoduli[5](8)^\odd$ 
lifted to $\barmoduli[{5,1}]$ using the zero of $\omega$. Then $\bC$ does not intersect the divisor 
$\Nfold^1_{5,4}(1)$. 
\end{prop}
\par
\begin{proof}
Suppose that $(X, \omega)$ is contained in the intersection of $\bC$ with $\Nfold^1_{5,4}(1)$. 
Note that $X$ is not hyperelliptic, as this component and the hyperelliptic component are disjoint
and by Proposition~\ref{prop:Teichhypdisj}. 
Recall that $\Nfold^1_{5,4} (1)\subset \barmoduli[{5,1}]$ parameterizes curves 
that admit a $g^1_4$ given by $3p+q$, where $p$ is the marked point and $q$ is a random point. 
Then it implies that $h^0(\OO_X(3p+q)) = 2$ and $h^0(\OO_X(4p)) = 1$ by Clifford's theorem, hence 
we have $q\neq p$. 
By Riemann-Roch, we have $h^0(\OO_X(5p-q)) = 2$. Since 
$h^0(\OO_X(5p)) = 2 = h^0(\OO_X(5p-q)) $, it implies that 
$q$ is a base point of $|\OO_X(5p)|$, which is impossible. 
\end{proof}

\begin{proof}[{Proof of Theorem~\ref{thm:g5main}, Case ${\omoduli[5]}(8)^\odd$}]
By the proposition we have 
$$\bC\cdot \Nfold^1_{5,4}(1) = 0.$$ 
It now suffices to plug the result of Proposition~\ref{prop:intomega} 
with $m_1=8$ into the divisor class of $\Nfold^1_{5,4}(1)$ in \eqref{eq:Nfold} to obtain the desired numbers.
\end{proof}

\subsection{The stratum $\omoduli[5](5,3)$} \label{sec:53}

In the case $\omoduli(5,3)$ we have  
$$L_{(5,3)} = 9/4, \quad s_{(5,3)}  = 209/27, \quad c_{(5,3)}  = 209/144.$$
\par
\begin{prop} \label{prop:53noint}
Let $C$ be a \teichm curve generated by a flat surface 
$(X,\omega) \in \omoduli[5](5,3)$,
lifted to $\barmoduli[{5, 2}]$ by the zeros of $\omega$. Then $\bC$
does not intersect the divisor $BN^1_{4,(1,2)}$.
\end{prop}
\par
\begin{proof}
Note that the degenerate fibers of the family over $\bC$ are either irreducible
or consist of two components connected by an odd number ($\geq 3$) of nodes
by Proposition~\ref{cor:degfibers}.
Moreover by Proposition~\ref{prop:Teichhypdisj} the degenerate fibers are not hyperelliptic.
Consequently the dual graph of $X$ is three-connected and 
$\omega_X$ is very ample by Proposition~\ref{prop:dualgraph}.
\par
Suppose that contrary to the claim, $(X,p,q)$ is contained in the intersection of $\bC$ and $BN^1_{4,(1,2)}$, 
i.e. $h^0(\OO_X(2p+q+r)) = 2$ for some $r\in X$. Since $\omega_X \sim \OO_X(5p + 3q)$, by Riemann-Roch 
we have $h^0(\OO_X(3p+2q - r)) = 2$. For $r = p$ or $q$, these equalities reduce to 
$h^0(\OO_X(2p+2q)) = 2$ and $h^0(\OO_X(3p+q)) = 2$. If $r\neq p, q$, then $r$ is not a base point 
of $|\OO_X(3p+2q)|$, hence $h^0(\OO_X(3p+2q)) = 3$ and 
$h^0(\OO_X(2p+q)) = 2$. Then we still have $h^0(\OO_X(2p+2q)) = h^0(3p+q) \geq 2$, hence they are equal 
to $2$ by Clifford's theorem.  In any case, $3p+q$ and $2p+2q$ span two different planes with the 
corresponding contact orders at $p$ and $q$ to the canonical image of $X$ in $\PP^4$. 
The two planes contain a common line spanned by $p, q$ whose intersection with $X$ is $2p+q$. 
By Geometric Riemann-Roch, $2p+q$ gives rise to a $g^1_3$, hence $X$ is trigonal. 
\par
For a trigonal genus 5 curve $X$ with $\omega_X$ very ample, as we have seen in Section~\ref{sec:8even}, 
its canonical image is contained in a cubic scroll surface in $\PP^4$. 
The residual $g^2_5$ given by $3p+2q$ maps $X$ to a plane quintic $Y$, whose image differs 
from $X$ at a double point $u$ (like a node or cusp), for the arithmetic genus of $Y$ is 6. 
The unique $g^1_3$ on $X$ is given by intersections of lines passing through $u$ with $Y$ (subtracting $2u$ from the base locus). But $2p+q$ is contained in the $g^1_3$, 
and the line spanned by $p,q$ has contact order 3 at $p$ and 2 at $q$ to $Y$, hence $u$ must be $p$ or $q$ by B\'{e}zout. For $u = q$, subtracting $2u$ from $3p+2q$, we know that 
$3p$ is also in the $g^1_3$, hence $3p \sim 2p+ q$, $p\sim q$, impossible. For $u = p$, we have $p+2q$ is in the $g^1_3$. Hence $p+ 2q\sim 2p+q$, which implies $p\sim q$ and this is also impossible. 
\end{proof}
\par
\begin{proof}[{Proof of Theorem~\ref{thm:g5main}, Case ${\omoduli[5]}(5,3)$}]
Proposition~\ref{prop:53noint} says that $\bC\cdot BN^1_{4,(1,2)} = 0$ for a \teichm curve 
$C$ in this stratum. Using the divisor class of $BN^1_{4,(1,2)}$ in \eqref{eq:Nfold2} together with Proposition~\ref{prop:intomega}, 
the result follows immediately.   
\end{proof}

\subsection{The stratum $\omoduli[5](6,2)^{\odd*}$} \label{sec:62}

In the case $\omoduli[5](6,2)^\odd$ we have  
$$L_{(6,2)^\odd} = 46/21, \quad s_{(6,2)^\odd} =  176/23, \quad c_{(6,2)^\odd} = 209/144. $$
\par
Based on numerical values on individual \teichm curves, we believe that the sum of 
Lyapunov exponents is non-varying in this stratum. But we have not discovered a divisor to carry out the desired disjointness argument. 
\par

\subsection{Examples of square-tiled surfaces in $g=5$ and $g=6$} \label{sec:5vary}

In this section we list examples of square-tiled surfaces in $g=5$ to 
justify that the sum of Lyapunov exponents in the remaining strata
is indeed varying.
\par
In the stratum $\omoduli[5](7,1)$ varying sum can be checked using the square-tiled surface
\begin{equation} \label{eq:ST71}
(\pi_r = (1 2 3 4 5 6 7 8 9\,10), \pi_u=(1 5 9 6)(2 4 7\,10)).
\end{equation}
In the stratum $\omoduli[5](5,2,1)$ varying sum can be checked using the square-tiled surface
\begin{equation} \label{eq:ST521}
(\pi_r = (1 2 3 4 5 6 7 8 9\,10\, 11), \pi_u=(1\, 11)(23)(46)(7 9)).
\end{equation}
In the stratum $\omoduli[5](5,1,1,1)$ varying sum can be checked using the square-tiled surface
\begin{equation} \label{eq:ST5111}
(\pi_r = (1 2 3 4 5 6 7 8 9\,10\, 11\, 12), \pi_u=(1\, 12)(23)(46)(8 10)).
\end{equation}
In the stratum $\omoduli[5](4,4)^\even$ varying sum can be checked using the square-tiled surface
\begin{equation} \label{eq:ST44even}
(\pi_r = (1 2 3 4 5 6 7 8 9\,10), \pi_u=(1 \, 10)(2 9)(3568)).
\end{equation}
In the stratum $\omoduli[5](4,4)^\odd$ varying sum can be cross-checked
using, besides the hyperelliptic locus, the square-tiled surface
\begin{equation} \label{eq:ST44odd}
(\pi_r = (1 2 3 4 5 6 7 8 9\,10), \pi_u=(1\, 10)(2 3)(5 6)(7 8)).
\end{equation}
\par
In the stratum $\omoduli[5](4,3,1)$ varying sum can be checked using the square-tiled surface
\begin{equation} \label{eq:ST431}
(\pi_r = (1 2 3 4 5 6 7 8 9\,10 \, 11 \, 12), \pi_u=(11\, 1 9 4)(2\, 10 \, 3 5 6)(7\, 12)).
\end{equation}
In the stratum $\omoduli[5](3,2,2,1)$ varying sum can be checked using the square-tiled surface
\begin{equation} \label{eq:ST3221}
(\pi_r = (1 2 3)(4 5 6 7 8 9\,10 \, 11 \, 12), \pi_u=(1\, 11)(10\, 5\, 13)(2 7)).
\end{equation}
In the stratum $\omoduli[5](3,2,1,1,1)$ varying sum can be checked using the square-tiled surface
\begin{equation} \label{eq:ST32111}
(\pi_r = (1 2 3 4 5 6 7 8 9\,10 \, 11 \, 12), \pi_u=(1\, 12)(24)(57)(8\, 10)).
\end{equation}
In the stratum $\omoduli[5](3,1,1,1,1,1)$ varying sum can be checked using the square-tiled surface
\begin{equation} \label{eq:ST31to5}
(\pi_r = (1 2 3 4 5 6 7 8 9\,10 \, 11 \, 12 \, 13),  \pi_u=(1\,14)(24)(6 8)(10\, 12)).
\end{equation}
\par
To indicate that the phenomenon of non-varying sum of Lyapunov exponents
is restricted to low genus and special loci, such as e.g.\ the hyperelliptic
locus, we show that already in $g=6$ the best candidates fail.
\par
\begin{prop}\label{prop:6vary}
For $g=6$ the sum of Lyapunov exponents is varying in the strata
$\omoduli[6](10)^\odd$ and $\omoduli[6](10)^\even$. 
\end{prop}
\par
\begin{proof}
For $\omoduli[6](10)^\odd$ the sum for the measure supported on the
whole stratum is $L_{(10)^\odd} = \frac{82680540070}{35169130909} $ using \cite{emz}, but the square-tiled surface
\begin{equation} \label{eq:10odd}
(\pi_r = (1 2 3 4 5 6 7 8 9\,10 \, 11), \pi_u=(1 3 5 7 9\,11))
\end{equation}
provides an example with $L(C) = \frac{3166}{1375}$. In the even case $L_{(10)^\even} =
\frac{9085753953118}{3770001658049}$ but the square-tiled surface
\begin{equation} \label{eq:10even}
(\pi_r = (1 2 3 4 5 6 7 8 9\,10 \, 11), \pi_u= (1 5 7 9\,11)(2 4))
\end{equation}
gives an example with $L(C) = \frac{244729}{101893}$. 
\end{proof}
\par

\section{Hyperelliptic strata and moduli spaces of pointed curves} \label{sec:hyp}
Using \teichm curves in the hyperelliptic strata we reverse our engine to present an application 
for the geometry of moduli spaces of pointed curves. 
\par
In the study of the geometry of a moduli space, a central question is to ask about the extremality of a divisor class, e.g. if it has non-negative intersection numbers
with various curve classes on the moduli space. Now consider the moduli space $\barmoduli[g,1]$ of genus $g$ curves with one marked point. Define a divisor class 
$$D_1 = 4g(g-1) \omega_\rel   - 12\lambda + \delta, $$
where $\delta$ is the total boundary class.  
Let $\cX \to B$ be a complete one-dimensional family of stable one-pointed curves with smooth generic fibers. Harris \cite[Theorem 1]{harrisfamily}
showed that $D_1\cdot B$ is always non-negative and asked further if this is optimal, i.e. if there exists such a family $B$ satisfying $D_1\cdot B = 0$. 
The reader may also refer to \cite[(6.31), (6.34)]{harrismorrison} for an expository explanation. Define another divisor class on the moduli space $\barmoduli[g,2]$ of genus $g$ curves with two marked points: 
$$ D_2 = (g^2-1)(\psi_{1} + \psi_{2}) - 12\lambda + \delta, $$
where $\psi_i$ is the first Chern class of the cotangent line bundle associated to the $i$-th marked point. 
By a completely analogous argument as in \cite{harrisfamily}, one easily checks that $D_2$ has non-negative intersection with 
any complete one-dimensional family of stable two-pointed curves with smooth generic fibers. Similarly one can ask if this is optimal. 
Below we show that in both cases the zero-intersection can be attained. 
\par
\begin{theorem}\label{thm:extremal}
Let $C_1, C_2$ be \teichm curves in generated by flat surfaces in $\omoduli(2g-2)^\hyp$ and $\omoduli(g-1,g-1)^\hyp$, lifted to $\barmoduli[g,1]$ and $\barmoduli[g,2]$ using the zeros of Abelian differentials, 
respectively. Then we have 
$$\bC_1\cdot D_1 = 0, $$
$$ \bC_2 \cdot D_2 = 0. $$
\end{theorem}
\par
\begin{proof}
By the value of $s(C_1)$ given in Corollary~\ref{cor:hypvalues} and Proposition~\ref{prop:intomega}, we obtain that 
$$ \frac{\bC_1\cdot \lambda}{\bC_1\cdot \omega_\rel} = g^2, $$
$$\frac{\bC_1\cdot \delta}{\bC_1\cdot \omega_\rel} = 4g(2g+1). $$
Plugging them into the intersection $\bC_1\cdot D_1$, an elementary calculation shows that $\bC_1\cdot D_1 = 0$. 
\par
Similarly we have 
$$ \frac{\bC_2\cdot \lambda}{\bC_2\cdot (\psi_{1}+\psi_{2})} = \frac{g(g+1)}{4}, $$
$$\frac{\bC_2\cdot \delta}{\bC_2\cdot (\psi_{1}+\psi_{2})} = (g+1)(2g+1). $$
One easily checks that $\bC_2\cdot D_2 = 0$. 
\end{proof}

Since square-tiled surfaces in a stratum correspond to 'lattice points' under the period coordinates, the union of all such \teichm curves $C_1, C_2$ forms a
Zariski dense subset in the hyperelliptic locus. Hence they provide infinitely many solutions to the above question. We finally remark that the positivity of $D_1$ as well as not being strictly ample has a transparent geometric explanation, pointed out to us by one of the referees. It is proportional to the pullback of the theta line bundle from the universal Jacobian over $\barmoduli[g,1]$ via the map $(C, p) \mapsto [\cO((2g-2)p) \otimes \omega_C^{*}]$ by Morita \cite[Theorem 1.6]{morita}.

\bibliography{my}
%

\end{document}